\input ./style/arxiv-general.cfg
\documentclass[seceqn,MSNbibl,number,citesort,dvips]{arxbj}
\makeatletter
   \@ifpackageloaded{graphicx}{}{\usepackage{graphicx}}
\makeatother
\usepackage{upgreek}

\aid{0}
\volume{21}
\issue{3}
\pubyear{2015}
\firstpage{1575}
\lastpage{1599}
\doi{10.3150/14-BEJ614} 

\makeatletter
\newcommand{\rrvert}{\vert}
\newcommand{\llvert}{\vert}
\newtheorem{theorem}{Theorem}[section]

\newtheorem{proposition}[theorem]{Proposition}

\newproclaim{definition}[theorem]{Definition}
\newremark{remark}[theorem]{Remark}
\newproclaim{example}[theorem]{Example}

\newcommand{\eqref}[1]{(\ref{#1})}
\newcommand{\BX}{\mathbf{X}}
\newcommand{\by}{\mathbf{y}}
\newcommand{\reals}{{\mathbb R}}
\newcommand{\bbr}{\reals}
\newcommand{\bbn}{{\mathbb N}}
\newcommand{\vep}{\varepsilon}
\newcommand{\bbs}{\mathbb{S}}

\newcommand{\SaS}{S$\alpha$S}

\newcommand{\lst}{\stackrel{\mathrm{st}}{\leq}}
\newcommand{\eid}{\stackrel{d}{=}}

\newcommand{\one}{\mathbf{1}}

\makeatother

\begin{document}
\begin{frontmatter}

\title{Maxima of long memory stationary symmetric $\alpha$-stable processes,
and self-similar processes with stationary max-increments}
\runtitle{Limit theory for partial maxima}

\begin{aug}
\author[1]{\inits{T.}\fnms{Takashi}~\snm{Owada}\thanksref{1}\ead[label=e1]{takashiowada@ee.technion.ac.il}} \and
\author[2]{\inits{G.}\fnms{Gennady}~\snm{Samorodnitsky}\corref{}\thanksref{2}\ead[label=e2]{gs18@cornell.edu}}
\address[1]{Faculty of Electrical Engineering,
Technion, Haifa, Israel 32000.\\ \printead{e1}}
\address[2]{School of Operations Research and Information Engineering,
and Department of Statistical Science,
Cornell University,
Ithaca, NY 14853, USA. \printead{e2}}
\end{aug}

\received{\smonth{7} \syear{2013}}
\revised{\smonth{12} \syear{2013}}

%
\begin{abstract}
We derive a functional limit theorem for the partial maxima process
based on a long memory stationary $\alpha$-stable process. The length
of memory in the stable process is parameterized by a certain
ergodic-theoretical parameter in an integral representation of the
process. The limiting process is no longer a classical extremal
Fr\'echet process. It is a self-similar process with
$\alpha$-Fr\'echet
marginals, and it has stationary max-increments, a property which
we introduce in this paper. The functional limit theorem is
established in the space $D[0,\infty)$ equipped with the Skorohod
$M_1$-topology; in certain special cases the topology can be
strengthened to the Skorohod $J_1$-topology.
\end{abstract}

%

\end{frontmatter}

\section{Introduction} \label{sec:intro}

The asymptotic behaviour of the partial maxima sequence $M_n=\max_{1
\leq
k \leq n} X_k$, $n=1,2,\ldots$ for an i.i.d. sequence
$(X_1,X_2,\ldots)$ of random
variables is the subject of the classical extreme value theory, dating
back to Fisher and Tippett \cite{fisher:tippett:1928}. The basic result
of this theory
says that only three one-dimensional distributions, the Fr\'echet
distribution, the Weibull distribution and the Gumbel distribution,
have a max-domain of attraction. If $Y$ has one of these three
distributions, then for a distribution in its domain
of attraction, and a sequence of i.i.d. random variables with
that distribution,
%
\begin{equation}
\label{e:1din.conv} \frac{M_n-b_n}{a_n}\quad \Rightarrow\quad Y
\end{equation}
for properly chosen sequences $(a_n)$, $(b_n)$; see, for example,
Chapter~1 in
Resnick \cite{resnick:1987} or Section~1.2 in
de~Haan and Ferreira \cite{dehaan:ferreira:2006}. Under the same
max-domain of attraction
assumption, a~functional version of \eqref{e:1din.conv} was
established in Lamperti \cite{lamperti:1964}: with the same sequences $(a_n)$,
$(b_n)$ as in \eqref{e:1din.conv},
%
\begin{equation}
\label{e:funct.conv} \biggl( \frac{M_{\lfloor nt\rfloor}-b_n}{a_n},
 t\geq0 \biggr) \quad\Rightarrow\quad \bigl(Y(t),
t \geq0 \bigr)
\end{equation}
for a nondecreasing right continuous process $(Y(t), t\geq0)$, and
the convergence is weak convergence in the Skorohod $J_1$-topology on
$D[0,\infty)$. The limiting process is often called {\it the extremal
process}; its properties were established in
Dwass \cite{dwass:1964,dwass:1966} and Resnick and Rubinovitch \cite
{resnick:rubinovitch:1973}.

Much of the more recent research in extreme value theory concentrated
on the case when the underlying sequence $(X_1,X_2,\ldots)$ is
stationary, but may be dependent. In this case the extrema of the
sequence may cluster, and it is natural to expect that the limiting
results \eqref{e:1din.conv} and \eqref{e:funct.conv} will, in general,
have to be different. The extremes of moving average processes have
received special attention; see, for example, Rootz\'en \cite{rootzen:1978},
Davis and Resnick \cite{davis:resnick:1985} and Fasen \cite
{fasen:2005}. The extremes of the
$\operatorname{GARCH}(1,1)$ process were investigated in Mikosch and St\u{a}ric\u{a}
\cite{mikosch:starica:2000b}.
The classical work on the extremes of dependent
sequences is
Leadbetter \textit{et~al}. \cite{leadbetter:lindgren:rootzen:1983}; in some
cases this clustering
of the extremes can be characterized through the {\it extremal
index} (introduced, originally, in Leadbetter~\cite{leadbetter:1983}). The
latter is a number $0\leq\theta\leq1$. Suppose that a stationary
sequence $(X_1,X_2,\ldots)$ has this index, and let $(\tilde
X_1,\tilde X_2,\ldots)$ be an i.i.d. sequence with the same
one-dimensional marginal distributions as $(X_1,X_2,\ldots)$. If
\eqref{e:1din.conv} and \eqref{e:funct.conv} hold for the
i.i.d. sequence, then the corresponding limits will satisfy
$\tilde Y\eid\tilde Y(1)$, but the limit in \eqref{e:1din.conv}
for the dependent sequence $(X_1,X_2,\ldots)$ will satisfy $Y \eid
\tilde Y(\theta)$. In particular, the limit will be equal to zero if
the extremal index is equal to zero. This case can be viewed as that
of long range dependence in the extremes, and it has been mostly
neglected by the extreme value community. Long range dependence is,
however, an important phenomenon in its own right, and in this paper
we take a step towards understanding how long range dependence affects
extremes.

A random variable $X$ is said to have a regularly varying tail with
index $-\alpha$ for $\alpha>0$ if
\[
P(X>x) = x^{-\alpha} L(x),\qquad x>0 , %
\]
where $L$ is a slowly varying at infinity function, and the
distribution of any such random variable is in the max-domain of
attraction of the Fr\'echet distribution with the same parameter
$\alpha$; see, for example, Resnick \cite{resnick:1987}. Recall that
the Fr\'echet law
$F_{\alpha,\sigma}$ on $(0,\infty)$ with the tail index $\alpha$ and
scale $\sigma>0$
satisfies
%
\begin{equation}
\label{e:frechet} F_{\alpha,\sigma}(x) = \exp \bigl\{ -\sigma^\alpha
x^{-\alpha
} \bigr\},\qquad x>0 .
\end{equation}
Sometimes the term {\it$\alpha$-Fr\'echet} is used.
In this paper, we discuss the case of regularly varying tails and the
resulting limits in
\eqref{e:funct.conv}. The limits obtained in this paper belong to
the family of the so-called {\it Fr\'echet processes}, defined
below. We would like to emphasize that, even for stationary sequences
with regularly varying tails, non-Fr\'echet limits may appear in
\eqref{e:funct.conv}. We are postponing a detailed discussion of this
point to a future publication.

A stochastic process $(Y(t), t \in T)$ (on an arbitrary parameter
space $T$) is called a Fr\'echet process if for all $n \geq1$, $a_1,
\ldots, a_n >0$ and $t_1, \ldots, t_n \in T$, the weighted maximum
$\max_{1 \leq j \leq n} a_j Y(t_j)$ follows a Fr\'echet law as in
\eqref{e:frechet}. The best known Fr\'echet process is the extremal
Fr\'echet process obtained in the scheme \eqref{e:funct.conv} starting
with an i.i.d. sequence with regularly varying tails. The extremal
Fr\'echet process $ ( Y(t), t\geq0 )$ has finite-dimensional distributions defined by
%
\begin{eqnarray}
\label{e:extreme.frechet} \bigl( Y(t_1),Y(t_2),\ldots,
Y(t_n) \bigr) &\eid& \bigl( X^{(1)}_{\alpha,t_1^{1/\alpha}}, \max
\bigl( X^{(1)}_{\alpha,t_1^{1/\alpha}}, X^{(2)}_{\alpha,(t_2-t_1)^{1/\alpha}} \bigr),
\ldots,
\nonumber
\\[-8pt]
\\[-8pt]
&&\hphantom{\bigl(}{}\max \bigl( X^{(1)}_{\alpha,t_1^{1/\alpha}}, X^{(2)}_{\alpha,(t_2-t_1)^{1/\alpha}},
\ldots, X^{(n)}_{\alpha,(t_n-t_{n-1})^{1/\alpha}} \bigr) \bigr)
\nonumber
\end{eqnarray}
for all $n$ and $0\leq t_1<t_2<\cdots<t_n$. The different random
variables in the right-hand side of \eqref{e:extreme.frechet} are
independent, with $X^{(k)}_{\alpha,\sigma}$ having the Fr\'echet law
$F_{\alpha,\sigma}$ in \eqref{e:frechet}, for any $k=1,\ldots, n$.
The stationarity and independence of the max-increments of the
extremal Fr\'echet processes make it similar to the better known
L\'evy processes which have stationary and independent
sum-increments.
The structure of general Fr\'echet processes has been extensively
studied in the last several years. These processes were introduced in
Stoev and Taqqu \cite{stoev:taqqu:2005}, and their representations (as
a part of a
much more general context) were studied in
Kabluchko and Stoev \cite{kabluchko:stoev:2012}. Stationary Fr\'echet
processes (in particular, their ergodicity and mixing) were discussed
in Stoev \cite{stoev:2008}, Kabluchko \textit{et~al}. \cite
{kabluchko:schlather:dehaan:2009} and
Wang and Stoev \cite{wang:stoev:2010}.

In this paper, we concentrate on the maxima of stationary
$\alpha$-stable processes with $0<\alpha<2$. Recall that a random
vector $\BX$ in $\bbr^d$ is called $\alpha$-stable if for any $A$ and
$B>0$ we have
\[
A\BX^{(1)} + B\BX^{(2)} \eid \bigl(A^\alpha+B^\alpha
\bigr)^{1/\alpha} \BX+ \by, %
\]
where $\BX^{(1)}$ and $\BX^{(2)}$ are i.i.d. copies of $\BX$, and
$\by$ is a deterministic vector (unless $\BX$ is deterministic,
necessarily, $0<\alpha\leq2$). A stochastic process $(X(t), t \in
T)$ is called $\alpha$-stable if all of its finite-dimensional
distributions are $\alpha$-stable. We refer the reader to
Samorodnitsky and Taqqu \cite{samorodnitsky:taqqu:1994} for information
on $\alpha$-stable
processes. When $\alpha=2$, an $\alpha$-stable process is Gaussian,
while in the case $0<\alpha<2$, both the left and the right tails of a
(nondegenerate) $\alpha$-stable random variable $X$ are (generally)
regularly varying with exponent $\alpha$. That is,
\[
P(X>x) \sim c_+ x^{-\alpha},\quad\quad P(X<-x) \sim c_- x^{-\alpha} \quad\quad\mbox{as $x
\to\infty$} %
\]
for some $c_+, c_-\geq0$, $c_++c_->0$. That is, if
$(X_1,X_2,\ldots)$ is an i.i.d. sequence of $\alpha$-stable random
variables, then the i.i.d. sequence $(|X_1|,|X_2|,\ldots)$ satisfies
\eqref{e:1din.conv} and \eqref{e:funct.conv} with $a_n=n^{1/\alpha}$
(and $b_n=0$), $n\geq1$. Of course, we are not planning to study the
extrema of an i.i.d. $\alpha$-stable sequence. Instead, we will study
the maxima of (the absolute values of) a stationary $\alpha$-stable
process. The reason we have chosen to work with stationary $\alpha$-stable
processes is that their structure is very rich, and is also relatively
well understood. This will allow us to study the effect of that
structure on the limit theorems \eqref{e:1din.conv} and
\eqref{e:funct.conv}. We are specifically interested in the long range
dependent case, corresponding to the zero value of the extremal
index.

The structure of stationary symmetric $\alpha$-stable (\SaS) processes
has been
clarified in the last several years in the works of Jan Rosi\'nski;
see, for example, Rosi\'nski \cite{rosinski:1995,rosinski:2006}. The integral
representation of
such a process can be chosen to have a very special form. The class of
stationary \SaS\ processes we will investigate requires a
representation slightly more restrictive than the one generally
allowed. Specifically, we will consider discrete-time stationary
processes of the form
%
\begin{equation}
\label{e:underlying.proc} X_n = \int_E f \circ
T^n(x) \,\mathrm{d}M(x), \qquad n=1,2,\ldots,
\end{equation}
where $M$ is a \SaS\ random measure on a measurable space
$(E,\mathcal{E})$ with a $\sigma$-finite {\it infinite} control
measure $\mu$. The map $T\dvtx E
\to E$ is a measurable map that preserves the measure $\mu$. Further,
\mbox{$f \in L^{\alpha}(\mu)$}. See
Samorodnitsky and Taqqu \cite{samorodnitsky:taqqu:1994} for details on
$\alpha$-stable random
measures and integrals with respect to these measures. It is
elementary to check that a process with a
representation \eqref{e:underlying.proc} is, automatically,
stationary. Recall that any stationary \SaS\ process has a
representation of the form:
%
\begin{equation}
\label{e:general.integral} X_n = \int_E f_n(x)
\,\mathrm{d}M(x),\qquad n=1,2,\ldots,
\end{equation}
with
%
\begin{equation}
\label{e:general.rosinski} f_n(x) = a_n(x) \biggl(
\frac{\mathrm{d}\mu\circ T^{n}}{\mathrm{d}\mu}(x) \biggr)^{1/\alpha}
f\circ T^{n}(x),\quad\quad x \in E
\end{equation}
for $n=1,2,\ldots$\,, where $T\dvtx E \to E$ is a one-to-one map with both $T$
and $T^{-1}$ measurable, mapping the control measure $\mu$ into an
equivalent measure, and the sequence $(a_n)$ takes values $\pm1$
(and has the so-called cocycle property). Here $M$ is \SaS\ (and $f
\in L^{\alpha}(\mu)$). See Rosi\'nski \cite{rosinski:1995}.

In \eqref{e:underlying.proc} we assume, however,
that map $T$ is measure preserving. The main reason is that the
ergodic-theoretical notions we are using have been developed for
measure preserving maps. Indeed, it has been observed that the
ergodic-theoretical properties of the
map $T$, either in \eqref{e:underlying.proc} or in
\eqref{e:general.rosinski}, have a major impact on the memory of a
stationary $\alpha$-stable process. See,
for example, Surgailis \textit{et~al}. \cite
{surgailis:rosinski:mandrekar:cambanis:1993},
Samorodnitsky \cite{samorodnitsky:2004a,samorodnitsky:2005}, Roy \cite
{roy:2007},
Resnick and Samorodnitsky \cite{resnick:samorodnitsky:2004}, Owada and
Samorodnitsky \cite{owada:samorodnitsky:2012},
Owada \cite{owada:2013}. The
most relevant for this work is the result of
Samorodnitsky \cite{samorodnitsky:2004a}, who proved that,
if the map $T$ in \eqref{e:underlying.proc} or in
\eqref{e:general.rosinski} is conservative, then using the
normalization $a_n=n^{1/\alpha}$ ($b_n=0$) in \eqref{e:1din.conv}, as
indicated by the marginal tails, produces the zero limit, so the
partial maxima grow, in this case, strictly slower than at the rate of
$n^{1/\alpha}$. On the other hand, if the map $T$ is not conservative,
then the normalization $a_n=n^{1/\alpha}$ in \eqref{e:1din.conv} is
the correct one, and it leads to a Fr\'echet limit (we will survey
the ergodic-theoretical notions in the next section). Therefore, the
extrema of \SaS\ processes corresponding to conservative flows cluster
so much that the sequence of the partial maxima grows at a slower rate
than that indicated by the marginal tails. This case can be thought of
as indicating long range dependence. It is, clearly, inconsistent with
a positive extremal index.

The Fr\'echet limit obtained in \eqref{e:1din.conv} by
Samorodnitsky \cite{samorodnitsky:2004a} remains valid when the map
$T$ is
conservative (but with the normalization of a smaller order than
$n^{1/\alpha}$), as long as the map $T$ satisfies a certain additional
assumption. If one views the stationary $\alpha$-stable process as
a natural function of the Poisson points forming the random measure
$M$ in \eqref{e:general.integral} then, informally, this assumption
guarantees that only the largest Poisson point contributes,
distributionally, to the asymptotic behaviour of the partial maxima of
the process. In this paper, we restrict ourselves to this situation as
well. However, we will look at the limits obtained in the
much more informative functional scheme \eqref{e:funct.conv}. In this
paper, the assumption on the map $T$ will be expressed in terms of the
rate of growth of the so-called wandering rate sequence, which we define in
the sequel. We would like to emphasize that, when this wandering rate sequence
grows at a rate slower than the one assumed in this paper, new
phenomena seem to arise. Multiple Poisson points may contribute to the
asymptotic distribution of the partial maxima, and non-Fr\'echet limit
may appear in \eqref{e:funct.conv}. We leave a detailed study of this
to a
subsequent work.

In the next section, we provide the elements of the infinite ergodic
theory needed for the rest of the paper. In Section~\ref{sec:lim.processes} we introduce a new notion, that of a process
with stationary max-increments. It turns out that the possible limits
in the functional maxima scheme \eqref{e:funct.conv} (with $b_n=0$)
are self-similar with stationary max-increments. We discuss the
general properties of such processes and then specialize to the
concrete limiting process we obtain in the main result of the paper,
stated and proved in Section~\ref{sec:FLTPM}.

\section{Ergodic theoretical notions} \label{sec:ergodic}

In this section, we present some basic notation and notions of, mostly
infinite, ergodic theory used in
the sequel. The main references are Krengel \cite{krengel:1985},
Aaronson \cite{aaronson:1997}, and Zweim\"uller \cite{zweimuller:2009}.

Let $(E,\mathcal{E},\mu)$ be a $\sigma$-finite, infinite measure
space. We will say that $A = B$ mod $\mu$ if $A, B \in\mathcal{E}$
and $\mu(A\triangle B)=0$. For $f\in L^1(\mu)$ we will often write
$\mu(f)$ for the integral $\int f \,\mathrm{d}\mu$.

Let $T\dvtx E \to E$ be a measurable map preserving the measure $\mu$. The
sequence $(T^n)$ of iterates
of $T$ is called a {\it flow}, and the ergodic-theoretical properties of
the map and the flow are identified. A map $T$ is called {\it ergodic} if
any $T$-invariant set $A$ (i.e., a set such that $T^{-1}A = A$ mod
$\mu$) is trivial, that is, it satisfies $\mu(A)=0$ or $\mu(A^c)=0$. A map
$T$ is said to be {\it conservative} if
\[
\sum_{n=1}^{\infty} \one_A \circ
T^n = \infty \quad\quad\mbox{a.e. on } A %
\]
for any $A \in\mathcal{E}$, $0 < \mu(A) < \infty$; if $T$ is also
ergodic, then the restriction ``{\it on} $A$'' is not needed.

The {\it conservative part} of a measure-preserving $T$ is the largest
$T$-invariant subset $C$ of $E$ such that the restriction of $T$ to
$C$ is conservative. The set $D=E \setminus C$ is the {\it
dissipative part} of $T$ (and the decomposition $E=C\cup D$ is
called {\it the Hopf decomposition} of $T$).

The {\it dual operator} $\widehat{T}\dvtx  L^1(\mu) \to L^1(\mu)$ is
defined by
%
\begin{equation}
\label{e:define.dual} \widehat{T} f = \frac{\mathrm{d}
 (\nu_f \circ T^{-1} )}{\mathrm{d} \mu},\quad\quad f\in L^1(\mu) ,
\end{equation}
where $\nu_f$ is the signed measure $\nu_f(A) = \int_{A} f
\,\mathrm{d}\mu$, $A \in\mathcal{E}$. The dual operator satisfies the duality
relation
%
\begin{equation}
\label{e:dual.rel} \int_E \widehat{T} f\cdot g \,\mathrm{d}\mu= \int
_E f\cdot g \circ T \,\mathrm{d}\mu
\end{equation}
for $f\in L^1(\mu), g\in L^\infty(\mu)$. Note that
\eqref{e:define.dual} makes sense for any nonnegative measurable
function $f$ on $E$, and the resulting $\widehat{T} f$ is again a
nonnegative measurable
function. Furthermore, \eqref{e:dual.rel} holds for arbitrary
nonnegative measurable functions $f$ and $g$.

A conservative, ergodic and measure preserving map $T$ is said to be
{\it pointwise dual ergodic}, if there exists a normalizing sequence
$a_n \nearrow\infty$ such that
%
\begin{equation}
\label{e:pde} \frac{1}{a_n} \sum_{k=1}^n
\widehat{T}^k f \to\mu(f) \quad\quad\mbox{a.e. for every } f \in
L^1(\mu) .
\end{equation}
The property of pointwise dual ergodicity rules out invertibility of
the map $T$. Since the measure $\mu$ is infinite, choosing a
nonnegative function $f$ and using Fatou's lemma shows that only rates
$a_n=\mathrm{o}(n)$ are possible in pointwise dual ergodicity. Intuitively, as
will be seen
in \eqref{e:prop3.8.7} below, the longer time it takes the trajectory
of a point under the map $T$ to
return to a set of a finite positive measure, the smaller is the
normalizing sequence $(a_n)$.

Sometimes we require that for some functions the above convergence
takes place uniformly on a certain set. A set $A \in
\mathcal{E}$ with $0 < \mu(A) < \infty$ is said to be a {\it
uniform set} for a conservative, ergodic and measure preserving map
$T$, if there exist a normalizing sequence $a_n \nearrow\infty$ and a
nontrivial nonnegative measurable function $f \in L^1(\mu)$
(nontriviality means that $f$ is different from zero on a set of
positive measure) such that
%
\begin{equation}
\label{e:uniform_set} \frac{1}{a_n} \sum_{k=1}^n
\widehat{T}^k f \to\mu(f) \quad\quad\mbox {uniformly, a.e. on } A .
\end{equation}
If \eqref{e:uniform_set} holds for $f=\one_A$, the set $A$ is called
{\it a Darling--Kac set}. A conservative, ergodic and measure
preserving map $T$ is pointwise dual ergodic if and only if $T$ admits
a uniform set; see Proposition 3.7.5 in Aaronson \cite{aaronson:1997}. In
particular, it is legitimate to use the same normalizing
sequence $(a_n)$ both in (\ref{e:pde}) and (\ref{e:uniform_set}).

Let $A \in\mathcal{E}$ with $0 < \mu(A) < \infty$. The frequency of
visits to the set $A$ along the trajectory $( T^n x)$, $x\in E$, is
naturally related to the {\it wandering rate} sequence
%
\begin{equation}
\label{e:wanderingrate} w_n = \mu \Biggl(\bigcup
_{k=0}^{n-1} T^{-k}A \Biggr) .
\end{equation}
If we define the first entrance time to $A$ by
\[
\varphi_A(x) = \min \bigl\{ n \geq1\dvtx  T^n x \in A \bigr
\} %
\]
(notice that $\varphi_A < \infty$ a.e. on $E$ since $T$ is
conservative and
ergodic), then $w_n \sim\mu(\varphi_A < n)$ as $n\to\infty$. Since $T$
is also measure preserving, we have
$\mu(A \cap\{ \varphi_A > k \}) = \mu(A^c \cap\{ \varphi_A = k \})$
for $k \geq1$ (see, e.g., Zweim\"uller \cite{zweimuller:2009}). Therefore,
alternative expressions for the wandering rate sequence are
\[
w_n = \mu(A) + \sum_{k=1}^{n-1}
\mu \bigl(A^c \cap\{ \varphi_A = k \} \bigr) = \sum
_{k=0}^{n-1} \mu \bigl(A \cap\{
\varphi_A > k \} \bigr) . %
\]

Suppose now that $T$ is a pointwise dual ergodic map, and let $A$ be a
uniform set for $T$. It turns out that, under an assumption of regular
variation, there is a precise connection between the wandering rate
sequence $(w_n)$ and the normalizing sequence
$(a_n)$ in \eqref{e:pde} and \eqref{e:uniform_set}. Specifically, let
$RV_{\gamma}$ represent the class of regularly varying at infinity
sequences (or
functions, depending on the context) of
index $\gamma$. If either $(w_n) \in RV_{\beta}$ or $(a_n) \in
RV_{1-\beta}$ for some
$\beta\in[0,1]$, then
%
\begin{equation}
\label{e:prop3.8.7} a_n \sim\frac{1}{\Gamma(2-\beta)
 \Gamma(1+\beta)} \frac{n}{w_n} \quad\quad\mbox{as
} n \to\infty.
\end{equation}
Proposition 3.8.7 in Aaronson \cite{aaronson:1997} gives one direction
of this
statement, but the argument is easily reversed. The
normalizing sequence $(a_n)$ and the wandering rate sequence $(w_n)$
are both related to the frequency with which a uniform set $A$ is
visited along the trajectory $(T^nx)$ that starts in $A$.

We finish this section with a statement on distributional convergence
of the partial maxima for pointwise dual ergodic flows. It will be
used repeatedly in the proof of the main theorem. For a measurable
function $f$ on $E$ define
\[
M_n(f) (x) = \max_{1 \leq k \leq n}\bigl|f \circ
T^k(x)\bigr| ,\quad\quad x \in E, n \geq1 . %
\]

The proposition below involves weak convergence in the space
$D[0,\infty)$ equipped with two different topologies, the Skorohod
$J_1$-topology and the Skorohod $M_1$-topology, introduced in
Skorohod \cite{skorohod:1956}. The details could be found, for
instance, in
Billingsley \cite{billingsley:1999} (for the $J_1$-topology), and in
Whitt \cite{whitt:2002} (for the $M_1$-topology). See also Remark
\ref{rk:topologies}.

In the sequel, we will use the convention $\max_{k\in K}b_k = 0$ for
a nonnegative sequence $(b_n)$, if $K=\varnothing$.

%
\begin{proposition} \label{p:max.ergodic}
Let $T$ be a pointwise dual ergodic map on a $\sigma$-finite,
infinite, measure space $(E,\mathcal{E},\mu)$. We assume that the
normalizing sequence $(a_n)$ is regularly varying with exponent
$1-\beta$ for some $0 < \beta\leq1$. Let $A \in\mathcal{E}$, $0 <
\mu(A) < \infty$, be a uniform set for $T$. Define a probability
measure on $E$ by $\mu_n (\cdot) = \mu(\cdot\cap\{\varphi_A \leq
n\})
/ \mu(\{\varphi_A \leq n \})$.
Let $f\dvtx E \to\bbr$ be a measurable bounded function supported by the
set $A$, that is, $\operatorname{supp}(f) \subset A$. Let $\| f \|_{\infty} = \inf\{M
\dvtx  |f(x)| \leq M \mbox{ a.e. on } A \}$. Then
%
\begin{eqnarray}
\label{e:underlying.conv}&& \bigl( M_{\lfloor nt\rfloor}(f),
 0\leq t\leq1 \bigr) \nonumber\\[-8pt]\\[-8pt]
 &&\quad\Rightarrow\quad\|f
\|_\infty ( \one_{\{ V_{\beta} \leq t \}}, 0\leq t\leq1 )\quad\quad \mbox{in the
$M_1$-topology on $D[0,1] $,}\nonumber
\end{eqnarray}
where the law of the left-hand side is computed with respect to
$\mu_n$, and $V_{\beta}$ is a random variable defined on a probability
space $(\Omega^{\prime},\mathcal{F}^{\prime},P^{\prime})$ with
$P^{\prime}(V_{\beta} \leq x) = x^{\beta}$, $0 < x \leq1$.
If $f=\one_A$, then the convergence above takes place in the
$J_1$-topology as well.
\end{proposition}
%

\begin{remark} \label{rk:topologies}
It is not difficult to see why the weak convergence in
\eqref{e:underlying.conv} holds in the $J_1$-topology for indicator
functions, but only in the $M_1$-topology in general. Indeed, for
functions $f$ other than the indicator function, the limiting value of
$\|f\|_\infty$ may have an asymptotically non-vanishing probability of
being reached in multiple closely placed steps, which precludes the
$J_1$-tightness, since the $J_1$-modulus does not become small; see,
for example, Theorem 13.2 in Billingsley \cite{billingsley:1999}. One
can easily
construct (very general) examples of situations in which this
can be made precise. On
the other hand, if $f = \one_A$, then the limiting value is reached
by a single jump, matching the single jump in the limiting process,
which gives convergence in the $J_1$-topology.
\end{remark}
\begin{pf*}{Proof of Proposition~\ref{p:max.ergodic}}
For $0<\vep<1$, let $A_\vep= \{ x\in A\dvtx  |f(x)|\geq(1-\vep)\|
f\|_\infty\}$. Note that each $A_\vep$ is uniform since $A$ is
uniform. Clearly,
\[
(1-\vep)\| f\|_\infty\one_{ \{\varphi_{A_\vep}(x)\leq nt \} } \leq M_{\lfloor nt\rfloor}(f) (x)
\leq\| f\|_\infty\one_{ \{\varphi
_A(x)\leq
nt \} } \quad\quad\mu\mbox{-a.e.} %
\]
for all $n\geq1$ and $0\leq t\leq1$. Since for monotone functions
weak convergence in the $M_1$-topology is implied by convergence in
finite-dimensional distributions (see, e.g., Proposition 2 in
Avram and Taqqu \cite{avram:taqqu:1992}), we can use Theorem 3.2 in
Billingsley \cite{billingsley:1999} in a finite-dimensional situation.
The statement of the proposition will follow once we show that, for a
uniform set $B$ (which could be either $A$ or $A_\vep$) the law of
$\varphi_B/n$ under $\mu_n$ converges to the law of $V_{\beta}$. Let
$(w_n^{(B)})$ be the corresponding wandering rate sequence. Since
\eqref{e:prop3.8.7} holds for $(w_n^{(B)})$ with the same normalizing
constants $(a_n)$, we know that $w_n^{(B)}\sim w_n^{(A)}:= w_n$ as
$n\to\infty$. Therefore,
\[
\mu_n \biggl( \frac{\varphi_{B}}{n} \leq x \biggr) = \frac{\mu
(\varphi_{B}
\leq\lfloor nx \rfloor)}{\mu(\varphi_A \leq n)}
\sim \frac{w_{\lfloor nx \rfloor}^{(B)}}{w_n} \to x^{\beta} %
\]
for all $0 < x \leq1$, because the wandering rate sequence $(w_n)$ is
regularly varying with index $\beta$ by \eqref{e:prop3.8.7}.

Next, suppose that $f(x) = \one_A(x)$. In this case, $M_{\lfloor nt
\rfloor}(\one_A)(x) = \one_{\{ \varphi_A(x) \leq nt \}}$. An
application of the Skorohod embedding theorem tells us that on some
common probability space, the time of the jump of the process
$\one_{\{ \varphi_A(\cdot) \leq nt \}}$ converges a.s. to the time
of the jump of the process
$\one_{\{ V_{\beta} \leq t \}}$. This, in turn, implies
a.s. convergence of these processes in the space $D[0,1]$ in the
$J_1$-topology, hence their weak convergence in that topology.
\end{pf*}

\section{Self-similar processes with stationary
max-increments} \label{sec:lim.processes}

The limiting process obtained in the next section shares with any possible
limits in the functional maxima scheme \eqref{e:funct.conv} (with
$b_n=0$) two very specific properties, one of which is classical, and
the other is less so. Recall that a stochastic process $ ( Y(t),
t\geq0 )$ is called self-similar with exponent $H$ of
self-similarity if for any $c>0$
\[
\bigl( Y(ct), t\geq0 \bigr)\eid \bigl( c^HY(t), t\geq0 \bigr)
\]
in the sense of equality of finite-dimensional distributions. The
best known classes of self-similar processes arise in various versions
of a functional central limit theorem for stationary processes, and
they have an additional property of stationary increments. Recall that
a stochastic process $ ( Y(t), t\geq0 )$ is said to have
stationary increments if for any $r\geq0$
%
\begin{equation}
\label{e:stat.incr} \bigl( Y(t+r)-Y(r), t\geq0 \bigr) \eid \bigl( Y(t)-Y(0), t\geq 0
\bigr) ;
\end{equation}
see, for example, Embrechts and Maejima \cite{embrechts:maejima:2002} and
Samorodnitsky \cite{samorodnitsky:2006LRD}. In the context of the
functional limit
theorem for the maxima \eqref{e:funct.conv}, a different property
appears.

%
\begin{definition} \label{d:max.stat.def}
A stochastic process $( Y(t), t \geq0 )$ is said to have
stationary max-increments if for every $r \geq0$, there exists,
perhaps on an enlarged probability space, a
stochastic process $ ( Y^{(r)}(t), t \geq0  )$ such that
%
\begin{eqnarray}
\label{e:def.statmaxi} \bigl( Y^{(r)}(t), t \geq0 \bigr) &\stackrel{d} {=}
&\bigl( Y(t), t \geq0 \bigr) ,
\nonumber
\\[-8pt]\\[-8pt]
\bigl( Y(t+r), t \geq0 \bigr) &\stackrel{d} {=}& \bigl( Y(r) \vee
Y^{(r)}(t), t \geq0 \bigr) .\nonumber
\end{eqnarray}
\end{definition}

Notice the analogy between the definition \eqref{e:stat.incr} of
stationary increments (when $Y(0)=0$) and Definition
\ref{d:max.stat.def}. Since the operations of taking the maximum is not
invertible (unlike summation), the latter definition, by necessity, is
stated in terms of existence of the max-increment process $ (
Y^{(r)}(t), t \geq0  )$.

%
\begin{theorem} \label{t:max.lamperti}
Let $(X_1,X_2,\ldots)$ be a stationary sequence. Assume that for some
sequence \mbox{$a_n\to\infty$}, and a stochastic process $(Y(t), t\geq0)$
such that $P(Y(t)=Y(1))<1$ for $t\neq1$,
\[
\biggl( \frac{1}{a_n} M_{\lfloor nt\rfloor}, t\geq0 \biggr) \quad\Rightarrow\quad
\bigl(Y(t), t\geq0 \bigr) %
\]
in terms of convergence of finite-dimensional distributions. Then
$(Y(t), t\geq0)$ is self-similar with exponent $H>0$ of
self-similarity, and has stationary max-increments. Furthermore,
\mbox{$(Y(t), t\geq0)$} is continuous in probability. The sequence
$(a_n)$ is regularly varying with index $H$.
\end{theorem}
\begin{pf}
The facts that the limiting process $(Y(t), t\geq0)$ is
self-similar with exponent $H\geq0$ of
self-similarity, and that the
sequence $(a_n)$ is regularly varying with index $H$, follow from the
Lamperti theorem; see Lamperti \cite{lamperti:1962}, or Theorem 2.1.1 in
Embrechts and Maejima \cite{embrechts:maejima:2002}. The case $H=0$ is
ruled out by the
assumption that $P(Y(t)=Y(1))<1$ for $t\neq1$.
Lamperti's theorem is usually stated
and proved in the context of convergence in the situation when the
time is scaled by a parameter converging to infinity along the real
values, whereas in our situation the time scaling converges to
infinity along a discrete sequence of the integers. However, it is
easy to check that for maxima of stationary processes convergence
along a discrete sequence provides the same information as convergence
along all real values. Note, further, that for every $0\leq t_1<t_2$
and $n$ large enough,
\[
\frac{1}{a_n} ( M_{\lfloor nt_2\rfloor} - M_{\lfloor
nt_1\rfloor} ) \leq \frac{1}{a_n}
\max_{nt_1<i\leq nt_2} X_i \lst \frac{1}{a_n}
M_{\lfloor2n(t_2-t_1)\rfloor} %
\]
by the stationarity. Taking weak limits, we see that the difference
$Y(t_2)-Y(t_1)$ is nonnegative and
bounded stochastically by $Y (2(t_2-t_1) )$. Therefore, it
follows from the
self-similarity of $(Y(t), t\geq0)$ that it is continuous in
probability.

We check now the stationarity of the max-increments of the limiting
process. Let $r>0$, and $t_i>0, i=1,\ldots, k$, some $k\geq1$.
Write
%
\begin{equation}
\label{e:decompose} \frac{1}{a_n} M_{\lfloor n(t_i+r)\rfloor} = \frac{1}{a_n}
M_{\lfloor
nr\rfloor} \bigvee\frac{1}{a_n} \max_{nr<j\leq n(t_i+r)}
X_j,\quad\quad i=1,\ldots, k .
\end{equation}
By the assumption of the theorem and stationarity of the process
$(X_1,X_2,\ldots)$,
\[
\frac{1}{a_n} M_{\lfloor nr\rfloor} \quad\Rightarrow\quad
 Y(r), \quad\quad\biggl( \frac{1}{a_n}
\max_{nr<j\leq n(t_i+r)} X_j, i=1,\ldots, k \biggr) \quad\Rightarrow\quad
\bigl( Y(t_1),\ldots, Y(t_k) \bigr) %
\]
as $n\to\infty$. Since every weakly converging sequence is tight, and
a sequence with tight marginals is itself tight, we conclude that
\[
\biggl( \frac{1}{a_n} M_{\lfloor nr \rfloor}, \biggl( \frac{1}{a_n} \max
_{nr< j \leq n(t_i+r)} X_j, i=1,\ldots, k \biggr) \biggr)
\]
is a tight sequence. This tightness means that for every sequence
$n_m\to\infty$ there is a subsequence $n_{m(l)}\to\infty$ and a
$k$-dimensional random vector $ ( Y^{(r)}(t_1),\ldots,
Y^{(r)}(t_k) ) \stackrel{d}{=}  ( Y(t_1),\ldots,
Y(t_k) )$
such that as $l \to\infty$,
\begin{eqnarray*}
&&\biggl( \frac{1}{a_{n_{m(l)}}} M_{\lfloor{n_{m(l)}}r\rfloor}, \biggl( \frac{1}{a_{n_{m(l)}}} \max
_{n_{m(l)}r<j\leq n_{m(l)}(t_i+r)} X_j, i=1,\ldots, k \biggr) \biggr)
\\
&&\quad\Rightarrow \quad\bigl( Y(r), \bigl( Y^{(r)}(t_1),\ldots,
Y^{(r)}(t_k) \bigr) \bigr) . %
\end{eqnarray*}

Let now $\tau_i$, $i=1,2,\ldots$ be an enumeration of the rational
numbers in $[0,\infty)$. A diagonalization argument shows that there
is a sequence $n_m\to\infty$ and a stochastic process $ (
Y^{(r)}(\tau_i), i=1,2,\ldots )$ with $ ( Y^{(r)}(\tau_i),
i=1,2,\ldots
)\eid ( Y(\tau_i), i=1,2,\ldots )$ such that
%
\begin{eqnarray}
\label{e:rational.conv} &&\biggl( \frac{1}{a_{n_{m}}} M_{\lfloor{n_{m}}r\rfloor}, \biggl(
\frac{1}{a_{n_{m}}} \max_{n_{m}r<j\leq n_{m}(\tau_i+r)} X_j, i=1,2,\ldots
\biggr) \biggr) \nonumber
\\[-8pt]\\[-8pt]
&&\quad\Rightarrow \quad\bigl( Y(r), \bigl( Y^{(r)}(
\tau_i), i=1,2,\ldots \bigr) \bigr)\nonumber
\end{eqnarray}
in finite-dimensional distributions, as $m\to\infty$. We extend the
process $Y^{(r)}$ to the entire positive half-line by setting
\[
Y^{(r)}(t) = \frac{1}2 \Bigl( \lim_{\tau\uparrow t,\ \mathrm{rational}}
Y^{(r)}(\tau) + \lim_{\tau\downarrow t, \ \mathrm{rational}} Y^{(r)}(\tau)
\Bigr),\quad\quad t\geq0 . %
\]
The continuity in probability implies that this process is a version
of $(Y(t), t\geq0)$. This continuity in
probability, \eqref{e:rational.conv} and monotonicity imply that as $m
\to\infty$,
%
\begin{equation}
\label{e:full.conv} \biggl( \frac{1}{a_{n_{m}}} M_{\lfloor{n_{m}}r\rfloor}, \biggl(
\frac{1}{a_{n_{m}}} \max_{n_{m}r<j\leq n_{m}(t+r)} X_j, t\geq0 \biggr)
\biggr) \quad\Rightarrow\quad \bigl( Y(r), \bigl( Y^{(r)}(t), t\geq0 \bigr) \bigr)\nonumber
\end{equation}
in finite-dimensional distributions. Now the stationarity of
max-increments follows from \eqref{e:decompose},
\eqref{e:full.conv} and continuous mapping theorem.
\end{pf}

%
\begin{remark} \label{rk:sup.measures}
Self-similar processes with stationary max-increments arising in a
functional maxima scheme \eqref{e:funct.conv} are close in spirit to
the stationary self-similar extremal processes of
O{'}Brien \textit{et al}. \cite{obrien:torfs:vervaat:1990}, while
extremal processes themselves
are defined as random sup measures. A
random sup measure is, as its name implies, indexed by sets. They also
arise in a limiting maxima scheme similar to \eqref{e:funct.conv}, but
with a stronger notion of convergence. Every stationary self-similar
extremal processes trivially produces a self-similar process with
stationary max-increments via restriction to sets of the type $[0,t]$
for $t\geq0$, but the connection between the two objects remains
unclear. Our limiting process in Theorem
\ref{t:main.max} below can be extended to a stationary self-similar
extremal processes, but the extension is highly nontrivial, and will
not be pursued here.\vadjust{\goodbreak}
\end{remark}

It is not our goal in this paper to study in details the properties of
self-similar processes with stationary max-increments, so we restrict
ourselves to the following basic result.

\begin{proposition} \label{pr:sssmaxi}
Let $ ( Y(t), t \geq0  )$ be a nonnegative self-similar
process with
stationary max-increments, and exponent $H$ of
self-similarity. Suppose $ (Y(t), t \geq0  )$ is not
identically zero. Then
$H\geq0$, and the following statements hold.
\begin{enumerate}
 \item[(a)]  If $H=0$, then $Y(t)=Y(1)$ a.s. for every $t>0$.

 \item[(b)]  If $0<EY(1)^p<\infty$ for some $p>0$, then $H\leq1/p$.

 \item[(c)]  If $H>0$, $ ( Y(t), t \geq0  )$ is continuous in
probability.
\end{enumerate}
\end{proposition}
\begin{pf}
By the stationarity of max-increments, $Y(t)$ is stochastically
increasing with $t$. This implies that $H\geq0$.

If $H=0$, then $Y(n)\eid Y(1)$ for each $n=1,2\ldots$\,. We use
\eqref{e:def.statmaxi} with $r=1$. Using $t=1$ we see that, in the
right-hand side of \eqref{e:def.statmaxi}, $Y(1)=Y^{(1)}(1)$
a.s. Since $Y^{(1)}(n)\geq Y^{(1)}(1)$ a.s., we conclude, using $t=n$
in the
right-hand side of \eqref{e:def.statmaxi}, that $Y(1)=Y^{(1)}(n)$ a.s.
for each $n=1,2,\ldots$\,. By monotonicity, we
conclude that the process $ ( Y^{(1)}(t), t \geq0  )$,
hence also the process $ ( Y(t), t \geq0  )$,
is a.s. constant on $[1,\infty)$ and then, by self-similarity, also on
$(0,\infty)$.

Next, let $p>0$ be such that $0<EY(1)^p<\infty$. It follows from
\eqref{e:def.statmaxi} with $r=1$ that
\[
2^H Y(1)\eid Y(2)\eid\max \bigl( Y(1), Y^{(1)}(1) \bigr) .
\]
Therefore,
\[
2^{pH} EY(1)^p = E Y(2)^p = E \bigl[
Y(1)^p \vee Y^{(1)}(1)^p \bigr] \leq2
EY(1)^p. %
\]
This means that $pH\leq1$.

Finally, we take arbitrary $0<s<t$. We use \eqref{e:def.statmaxi} with
$r=s$. For every $\eta> 0$,
\begin{eqnarray*}
P \bigl( Y(t) - Y(s) > \eta \bigr) &=& P \bigl( Y(s) \vee Y^{(s)}(t-s) -
Y(s) > \eta \bigr)
\\
&\leq& P \bigl( Y^{(s)}(t-s) > \eta \bigr)
= P \bigl(
(t-s)^H Y(1) > \eta \bigr).
\end{eqnarray*}
Hence, continuity in probability.
\end{pf}

We now introduce a crucial object for the subsequent discussion,
which is the limiting process obtained in the main limit theorem
of Section~\ref{sec:FLTPM}. It has a somewhat deceptively simple
representation that we presently describe.

Let $\alpha>0$, and consider
the extremal Fr\'echet process $Z_{\alpha}(t), t\geq0$,
defined in \eqref{e:extreme.frechet}, with the scale $\sigma=1$. For
$0<\beta<1$, we define a new stochastic process by
%
\begin{equation}
\label{e:lim.process} Z_{\alpha, \beta}(t) = Z_{\alpha} \bigl(t^\beta
\bigr),\quad\quad t\geq0 .
\end{equation}
We will refer to this process as the {\it time scaled extremal Fr\'echet
process}.

The next proposition places this process in the general framework
introduced earlier in this section.

%
\begin{proposition} \label{pr:limproc.properties}
The process $Z_{\alpha, \beta}$ in \eqref{e:lim.process}
is self-similar with $H=\beta/\alpha$ and has stationary max-increments.
\end{proposition}
\begin{pf}
Since the extremal Fr\'echet process is self-similar with
$H=1/\alpha$, it is immediately seen that the process $Z_{\alpha,
\beta}$ is self-similar with $H=\beta/\alpha$.

To show the stationarity of max-increments, we start with a useful
representation of the extremal Fr\'echet process $Z_{\alpha}(t),
t\geq0$ in terms of the points of a Poisson random measure. Let
$ ((j_k,s_k) )$ be the points of a Poisson random measure on
$\bbr_+^2$
with mean measure $\rho_{\alpha} \times\lambda$, where
$\rho_{\alpha}(x,\infty) = x^{-\alpha}$, $x>0$ and $\lambda$ is the
Lebesgue measure on $\bbr_+$. Then an elementary calculation shows
that
\[
\bigl( Z_{\alpha}(t), t\geq0 \bigr) \eid \bigl( \sup \{ j_k\dvtx
s_k\leq t \}, t\geq0 \bigr) . %
\]
Therefore, $ ( Z_{\alpha,\beta}(t), t\geq0 ) \eid (
U_{\alpha,\beta}(t), t\geq0 ) $, where
%
\begin{equation}
\label{e:represent.U} U_{\alpha,\beta}(t) = \sup \bigl\{ j_k \dvtx
s_k \leq t^{\beta} \bigr\} ,\quad\quad t \geq0 .
\end{equation}

Given $r > 0$, we define
\[
U_{\alpha,\beta}^{(r)}(t) = \sup \bigl\{ j_k \dvtx
(t+r)^{\beta}-t^{\beta} \leq s_k \leq(t+r)^{\beta}
\bigr\} . %
\]
Since $0 < \beta< 1$, we have
\[
\bigl( (t_1+r)^{\beta}-t_1^{\beta},
(t_1+r)^{\beta} \bigr) \subset \bigl( (t_2+r)^{\beta}-t_2^{\beta},
(t_2+r)^{\beta} \bigr) %
\]
for $0\leq t_1<t_2$. The nested nature of these sets
implies that
\[
\bigl( U_{\alpha,\beta}^{(r)}(t), t \geq0 \bigr) \stackrel{d} {=}
\bigl( U_{\alpha,\beta}(t), t \geq0 \bigr) , %
\]
because only the obvious equality of the one-dimensional
distributions must be checked.
Furthermore, since $(t+r)^{\beta}-t^{\beta}\leq r^\beta$, we see that
\[
U_{\alpha,\beta}(t+r) = U_{\alpha,\beta}(r) \vee U_{\alpha,\beta
}^{(r)}(t)\quad\quad
\mbox{for all } t \geq0 . %
\]
This means that the process $U_{\alpha,\beta}$ has stationary
max-increments and, hence, so does the process~$Z_{\alpha,\beta}$.
\end{pf}

Note that the max-increment process $ ( U_{\alpha,\beta}^{(r)}(t)
)$ in the proof of Proposition \ref{pr:limproc.properties}
is not independent of the random variable
$U_{\alpha,\beta}(r)$ if $\beta<1$. The case $\beta=1$ corresponds to
the extremal Fr\'echet process, whose max-increments are both
stationary and independent.

It is interesting to note that, by part (b) of Proposition
\ref{pr:sssmaxi}, any $H$-self-similar process with stationary
max-increments and $\alpha$-Fr\'echet marginals, must satisfy $H\leq
1/\alpha$. The exponent $H=\beta/\alpha$ with $0<\beta\leq1$ of the
process $Z_{\alpha,\beta}$ (with $\beta=1$ corresponding to the
extremal Fr\'echet process $Z_{\alpha}$) covers the entire interval
$(0,1/\alpha]$. Therefore, the upper bound of part (b) of Proposition
\ref{pr:sssmaxi} is, in general, the best possible.

We finish this section by mentioning that an immediate conclusion from
\eqref{e:represent.U} is the following representation of
the time scaled extremal Fr\'echet process
$Z_{\alpha, \beta}$ on the interval $[0,1]$:
%
\begin{equation}
\label{e:represent.Z} \bigl( Z_{\alpha, \beta}(t), 0 \leq t \leq1 \bigr) \stackrel{d}
{=} \Biggl( \bigvee_{j=1}^{\infty}
\Gamma_j^{-1/\alpha} \one_{\{ V_j
\leq t
\}}, 0 \leq t \leq1 \Biggr)
,
\end{equation}
where $\Gamma_j$, $j=1,2,\ldots$\,, are arrival times of a unit rate
Poisson process on $(0,\infty)$, and $(V_j)$ are i.i.d. random
variables with $P(V_1 \leq x) = x^{\beta}$, $0 < x \leq1$,
independent of $(\Gamma_j)$.

\section{A functional limit theorem for partial maxima} \label{sec:FLTPM}

In this section, we state and prove our main result, a functional limit
theorem for the partial maxima of the discrete-time stationary
process $\BX=(X_1,X_2,\ldots)$
given in (\ref{e:underlying.proc}). Recall that $T$ is a conservative,
ergodic and
measure preserving map on a $\sigma$-finite, infinite, measure space
$(E,\mathcal{E},\mu)$. We will assume that $T$ is a pointwise dual
ergodic map with normalizing sequence $(a_n)$ that is regularly
varying with exponent $1-\beta$; equivalently, the wandering sequence
$(w_n)$ in \eqref{e:wanderingrate} is assumed to be regularly
varying with exponent $\beta$. Crucially, we will assume that
$1/2 < \beta< 1$. See Remark \ref{rk:other.beta} after the proof
of Theorem \ref{t:main.max} below.

Define
%
\begin{equation}
\label{e:b.n} b_n = \biggl( \int_E \max
_{1 \leq k \leq n} \bigl| f\circ T^n(x) \bigr|^{\alpha} \mu(\mathrm{d}x)
\biggr)^{1/\alpha}, \quad\quad n=1,2,\ldots.
\end{equation}
The sequence $(b_n)$ is known to play an important role in the rate of
growth of partial maxima of an $\alpha$-stable process of the type
\eqref{e:underlying.proc}. It also turns out to be a proper
normalizing sequence for our functional limit theorem. In
Samorodnitsky \cite{samorodnitsky:2004a} it was shown that, for a
canonical kernel
(\ref{e:general.rosinski}), if the map $T$ is
conservative, then the sequence $(b_n)$ grows at a rate strictly
slower than $n^{1/\alpha}$. The extra assumptions imposed in the current
paper will guarantee a more precise statement. We will prove that, in
fact, $(b_n) \in RV_{\beta/\alpha}$ and, more specifically,
%
\begin{equation}
\label{e:RV.exp.bn} \lim_{n\to\infty}\frac{b_n^\alpha}{ w_n}=\| f
\|_\infty
\end{equation}
(where $(w_n)$ is the wandering sequence).
This fact has an interesting message, because it explicitly shows
that the rate of growth of the partial maxima is determined both by
the heaviness of the marginal tails (through $\alpha$) and by the
length of memory (through $\beta$). Such a precise measure of the
length of memory is not present in Samorodnitsky \cite{samorodnitsky:2004a}.

In contrast, if the map $T$ has a nontrivial dissipative component,
then the sequence $(b_n)$ grows at the rate $n^{1/\alpha} $, and so
do the partial maxima of the stationary \SaS\ process; see
Samorodnitsky~\cite{samorodnitsky:2004a}. This is the limiting case of
the setup in
the present paper, as $\beta$ gets closer to $1$. Intuitively, the
smaller is $\beta$, the longer is the memory in the process.

The basic idea in the proof of our main result, Theorem
\ref{t:main.max} below, is similar to the idea in the proof of
Theorems 3.1 and 4.1 in Samorodnitsky \cite{samorodnitsky:2004a} and is
based on a
Poisson representation of the process and a ``single jump'' property;
see Remark \ref{rk:other.beta}.

We recall the tail constant of an $\alpha$-stable random variable
given by
\[
C_{\alpha} = \biggl( \int_0^{\infty}
x^{-\alpha} \sin x \,\mathrm{d}x \biggr)^{-1} = \cases{ (1-\alpha) /
 \bigl(
\Gamma(2-\alpha) \cos (\uppi \alpha/ 2) \bigr) &\quad \mbox{if } $\alpha\neq1$,
\cr
2 /
\uppi&\quad \mbox{if } $\alpha=1$; } %
\]
see Samorodnitsky and Taqqu \cite{samorodnitsky:taqqu:1994}.

%
\begin{theorem} \label{t:main.max}
Let $T$ be a conservative, ergodic and measure preserving map on a
$\sigma$-finite infinite measure space $(E,\mathcal{E},\mu)$. Assume
that $T$ is a pointwise dual ergodic map with normalizing
sequence $(a_n) \in RV_{1-\beta}$, $0\leq\beta\leq1$. Let $f\in
L^{\alpha}(\mu)\cap L^{\infty}(\mu)$, and assume that $f$ is
supported by
a uniform set $A$ for $T$, that is, $\operatorname{supp}(f) \subset A$. Let $\alpha
>0$. Then the sequence $(b_n)$
in \eqref{e:b.n} satisfies~\eqref{e:RV.exp.bn}.

Assume now that $0<\alpha<2$ and $1/2<\beta<1$. If $M$ is a S$\alpha$S
random measure on $(E,\mathcal{E})$ with
control measure $\mu$, then the stationary S$\alpha$S process $\BX$
given in (\ref{e:underlying.proc}) satisfies
%
\begin{equation}
\label{e:weak.conv} \biggl( \frac{1}{b_n} \max_{1 \leq k \leq\lfloor nt \rfloor}|X_k|,
t\geq0 \biggr) \quad\Rightarrow \quad\bigl( C_{\alpha}^{1/\alpha}
Z_{\alpha,
\beta
}(t), t\geq0 \bigr) \quad\quad\mbox{in } D[0,\infty)
\end{equation}
in the Skorohod $M_1$-topology. Moreover, if $f = \one_A$, then the
above convergence occurs in the Skorohod $J_1$-topology as well.
\end{theorem}

%
\begin{remark}
The functional limit theorem in Theorem \ref{t:main.max} above, once
again, involves weak convergence in two different topologies, that is,
the Skorohod $J_1$-topology and the Skorohod $M_1$-topology. The
issue is similar to that in Proposition \ref{p:max.ergodic}; see
Remark \ref{rk:topologies}.
\end{remark}
\begin{pf*}{Proof of Theorem \ref{t:main.max}}
We start with verifying \eqref{e:RV.exp.bn}. Obviously,
\[
b_n^{\alpha} \leq \| f \|_{\infty} \mu(
\varphi_A \leq n) , %
\]
and, recalling that $w_n \sim\mu(\varphi_A \leq n)$, we get the upper
bound
\[
\limsup_{n\to\infty}\frac{b_n^\alpha}{w_n}\leq\| f\|_\infty.
\]
On the other hand, take an arbitrary $\epsilon\in ( 0,\|
f \|_{\infty} )$. The set
\[
B_{\epsilon} = \bigl\{x \in A\dvtx  \bigl|f(x)\bigr| \geq \| f \|_{\infty} -
\epsilon \bigr\}  %
\]
is a uniform set for $T$. A lower bound for $b_n^{\alpha}$ is obtained
via the obvious inequality
\[
b_n^{\alpha} \geq \bigl( \| f \|_{\infty} - \epsilon \bigr) \mu
\Biggl( \bigcup_{j=1}^n T^{-j}
B_{\epsilon} \Biggr) . %
\]
Indeed, let $ ( w_n^{(\epsilon)} )$ be the corresponding wandering
rate sequence to the set $B_{\epsilon}$. As argued in Proposition \ref
{p:max.ergodic}, we know
that $w_n \sim w_n^{(\epsilon)} \sim\mu(\varphi_{B_{\epsilon}} \leq
n)$. Therefore,
\[
\liminf_{n\to\infty}\frac{b_n^\alpha}{ w_n} = \liminf
_{n\to\infty}\frac{b_n^\alpha}{\mu(\varphi_{B_{\epsilon}}
\leq n)} \geq \| f \|_{\infty} -
\epsilon. %
\]
Letting $\epsilon\to0$, we obtain \eqref{e:RV.exp.bn}.

Suppose now that $0<\alpha<2$ and $1/2<\beta<1$. We continue with
proving convergence in the finite-dimensional distributions in
\eqref{e:weak.conv}. Since for random elements in $D[0,\infty)$ with
nondecreasing sample paths, weak convergence in the $M_1$-topology is
implied by the finite-dimensional weak convergence, this will also
establish \eqref{e:weak.conv} in the sense of weak convergence in the
$M_1$-topology.

Fix $0 = t_0< t_1 < \cdots< t_d$, $d \geq1$. We may and will assume
that $t_d \leq1$. We use a series representation of the random vector
$(X_1, \ldots, X_n)$: with $f_k=f\circ T^k$, $k=1,2,\ldots$\,,
%
\begin{equation}
\label{e:series.max} (X_k, k=1,\ldots,n) \stackrel{d} {=} \Biggl(
b_n C_{\alpha}^{1/\alpha} \sum
_{j=1}^{\infty} \epsilon_j
\Gamma_j^{-1/\alpha} \frac{f_k(U_j^{(n)})}{\max_{1 \leq i \leq n}
|f_i(U_j^{(n)})|}, k=1,\ldots,n \Biggr) .
\end{equation}
Here $(\epsilon_j)$ are i.i.d. Rademacher random variables
(symmetric random variables with values $\pm
1$), $(\Gamma_j)$ are the arrival times of a unit rate Poisson
process on $(0,\infty)$, and $(U_j^{(n)})$ are i.i.d. $E$-valued random
variables with the common law $\eta_n$ defined by
%
\begin{equation}
\label{e:eta.n} \frac{\mathrm{d}\eta_n}{\mathrm{d}\mu}(x) =
\frac{1}{b_n^{\alpha}} \max
_{1 \leq k
\leq n} \bigl|f_k(x)\bigr|^{\alpha} ,\quad\quad x \in E .
\end{equation}
The sequences $(\epsilon_j)$, $(\Gamma_j)$, and $(U_j^{(n)})$ are
taken to be independent. We refer to Section~3.10 of
Samorodnitsky and Taqqu \cite{samorodnitsky:taqqu:1994} for series
representations of
$\alpha$-stable random vectors. The representation
\eqref{e:series.max} was also used in Samorodnitsky \cite
{samorodnitsky:2004a}, and
the argument below is structured similarly to the corresponding
argument ibid.

The crucial consequence of the assumption $1/2<\beta<1$ is that, in
the series representation~(\ref{e:series.max}), only the largest
Poisson jump will play an important role. It is shown in
Samorodnitsky \cite{samorodnitsky:2004a} that, under the assumptions
of Theorem
\ref{t:main.max}, for every $\eta>0$,
%
\begin{eqnarray}
\label{e:single.poisson.jump} \varphi_n(\eta) &\equiv& P \Biggl( \bigcup
_{k=1}^n \biggl\{ \Gamma_j^{-1/\alpha}
\frac{|f_k(U_j^{(n)})|}{\max_{1 \leq i \leq
n}|f_i(U_j^{(n)})|} > \eta \nonumber
\\[-8pt]\\[-8pt]
&&\hphantom{P \Biggl(\bigcup
_{k=1}^n \biggl\{}{}\mbox{for at least 2 different } j=1,2,\ldots \biggr\}
\Biggr) \to0\nonumber
\end{eqnarray}
as $n \to\infty$.

We will proceed in two steps. First, we will prove that
%
\begin{eqnarray}
\label{e:fidi.Malphabeta} &&\Biggl( \bigvee_{j=1}^{\infty}
\Gamma_j^{-1/\alpha} \frac{\max_{1
\leq
k \leq\lfloor nt_i \rfloor} |f_k(U_j^{(n)})|}{\max_{1 \leq k
\leq n} |f_k(U_j^{(n)})|}, i=1,\ldots,d \Biggr)
\nonumber
\\[-8pt]\\[-8pt]
&&\quad\Rightarrow\quad \bigl( Z_{\alpha, \beta}(t_i),
i=1,\ldots,d \bigr) \quad\quad\mbox{in
} \bbr_+^d .\nonumber
\end{eqnarray}
Next, we will prove that, for fixed $\lambda_1, \ldots, \lambda_d >
0$, for every $0<\delta<1$,
%
\begin{eqnarray}\label{e:upp.bdd.max}
&&P \Bigl( b_n^{-1} \max_{1 \leq k \leq\lfloor nt_i \rfloor}
|X_k| > \lambda_i, i=1,\ldots,d \Bigr)
\nonumber
\\[-8pt]\\[-8pt]
&&\quad\leq P \Biggl( C_{\alpha}^{1/\alpha} \bigvee
_{j=1}^{\infty} \Gamma_j^{-1/\alpha}
\frac{\max_{1 \leq k \leq\lfloor nt_i
\rfloor}|f_k(U_j^{(n)})|}{\max_{1 \leq k \leq
n}|f_k(U_j^{(n)})|} > \lambda_i(1-\delta), i=1,\ldots,d \Biggr) + \mathrm{o}(1)
\nonumber
\end{eqnarray}
and that
%
\begin{eqnarray}\label{e:lower.bdd.max}
&&P \Bigl( b_n^{-1} \max_{1 \leq k \leq\lfloor nt_i \rfloor}
|X_k| > \lambda_i, i=1,\ldots,d \Bigr)
\nonumber
\\[-8pt]\\[-8pt]
&&\quad\geq P \Biggl( C_{\alpha}^{1/\alpha} \bigvee
_{j=1}^{\infty} \Gamma_j^{-1/\alpha}
\frac{\max_{1 \leq k \leq\lfloor nt_i
\rfloor}|f_k(U_j^{(n)})|}{\max_{1 \leq k \leq
n}|f_k(U_j^{(n)})|} > \lambda_i(1+\delta), i=1,\ldots,d \Biggr) + \mathrm{o}(1)
.\nonumber
\end{eqnarray}
Since the Fr\'echet distribution is continuous, the weak convergence
\[
\Bigl( b_n^{-1} \max_{1 \leq k \leq\lfloor nt_i \rfloor}
|X_k|, i=1,\ldots,d \Bigr)\quad \Rightarrow\quad \bigl( Z_{\alpha, \beta}(t_i),
i=1,\ldots,d \bigr) \quad\quad\mbox{in } \bbr_+^d %
\]
will follow by taking $\delta$ arbitrarily small.

We start with proving \eqref{e:fidi.Malphabeta}. For $n=1,2,\ldots$\,,
$N_n=\sum_{j=1}^{\infty} \delta_{(\Gamma_j, U_j^{(n)})}$ is a Poisson
random measure on $(0,\infty)\times\bigcup_{k=1}^{n} T^{-k}A$
with mean measure $\lambda\times\eta_n$. Define a map $S_n\dvtx \bbr_+
\times\bigcup_{k=1}^{n} T^{-k}A \to\bbr_+^d$ by
\[
S_n(r,x) = r^{-1/\alpha} \bigl( M_n(f) (x)
\bigr)^{-1} \bigl( M_{\lfloor nt_1 \rfloor}(f) (x), \ldots, M_{\lfloor nt_d \rfloor}(f)
(x) \bigr) ,\qquad r > 0, x \in\bigcup_{k=1}^{n} T^{-k}A .
\]
Then, for $\lambda_1, \ldots, \lambda_d > 0$,
\begin{eqnarray*}
&&P \Biggl( \bigvee_{j=1}^{\infty}
\Gamma_j^{-1/\alpha} \frac{\max_{1
\leq k \leq\lfloor nt_i \rfloor} |f_k(U_j^{(n)})|}{\max_{1 \leq
k \leq n} |f_k(U_j^{(n)})|} \leq\lambda_i
, i=1,\ldots,d \Biggr)
\\
&&\quad= P \bigl[ N_n \bigl( S_n^{-1} \bigl(  (0,
\lambda_1  ] \times\cdots \times (0,\lambda_d]
\bigr)^c \bigr) = 0 \bigr]
\\
&&\quad= \exp \bigl\{ - (\lambda\times\eta_n) \bigl( S_n^{-1}
\bigl(  (0,\lambda_1] \times\cdots\times(0,\lambda_d]
\bigr)^c  \bigr) \bigr\}
\\
&&\quad= \exp \Biggl\{ - (\lambda\times\eta_n) \Biggl\{(r,x)\dvtx  \bigvee
_{j=1}^d \lambda_j^{-\alpha}
\frac{ (M_{\lfloor nt_j
\rfloor}(f)(x) )^{\alpha}}{ (M_n(f)(x) )^{\alpha}} > r \Biggr\} \Biggr\}
\\
&&\quad= \exp \Biggl\{ -b_n^{-\alpha} \int_E
\bigvee_{j=1}^d \lambda
_j^{-\alpha
} M_{\lfloor nt_j \rfloor}(f)^{\alpha} \,\mathrm{d}\mu \Biggr\}
.
\end{eqnarray*}
We use \eqref{e:RV.exp.bn} and the weak convergence in
Proposition \ref{p:max.ergodic} to obtain
\begin{eqnarray*}
&&b_n^{-\alpha} \int_E \bigvee
_{j=1}^d \lambda_j^{-\alpha}
M_{\lfloor
nt_j \rfloor}(f)^{\alpha} \,\mathrm{d}\mu
 \sim\| f \|_{\infty}^{-1}
\int_E \bigvee_{j=1}^d
\lambda _j^{-\alpha} M_{\lfloor nt_j
\rfloor}(f)^{\alpha} \,\mathrm{d}
\mu_n %
\\
&&\quad\to\quad\int_{\Omega^{\prime}} \bigvee_{j=1}^d
\lambda_j^{-\alpha} \one_{\{ V_{\beta} \leq t_j \}} \,\mathrm{d}P^{\prime} =
\sum_{i=1}^d \bigl(t_i^{\beta}
- t_{i-1}^{\beta} \bigr) \Biggl( \bigwedge
_{j=i}^d \lambda_j
\Biggr)^{-\alpha}. %
\end{eqnarray*}
Therefore,
\begin{eqnarray*}
&&P \Biggl( \bigvee_{j=1}^{\infty}
\Gamma_j^{-1/\alpha} \frac{\max_{1
\leq k \leq\lfloor nt_i \rfloor} |f_k(U_j^{(n)})|}{\max_{1 \leq k
\leq
n} |f_k(U_j^{(n)})|} \leq\lambda_i
, i=1,\ldots,d \Biggr)
\\
&&\quad\to\quad\exp \Biggl\{ - \sum_{i=1}^d
\bigl(t_i^{\beta} - t_{i-1}^{\beta} \bigr)
\Biggl( \bigwedge_{j=i}^d
\lambda_j \Biggr)^{-\alpha} \Biggr\} = P \bigl(
Z_{\alpha, \beta}(t_i) \leq\lambda_i, i=1,\ldots,d \bigr)
.
\end{eqnarray*}
The claim (\ref{e:fidi.Malphabeta}) has, consequently, been proved.

We continue with the statements (\ref{e:upp.bdd.max}) and (\ref
{e:lower.bdd.max}).
Since the arguments are very similar, we only prove (\ref{e:upp.bdd.max}).
Let $K \in\bbn$ and $0 < \epsilon< 1$ be constants so that
\[
K+1 > \frac{4}{\alpha} \quad\mbox{and}\quad \delta- \epsilon K > 0 . %
\]
Then
\begin{eqnarray*}
&&P \Bigl( b_n^{-1} \max_{1 \leq k \leq\lfloor nt_i \rfloor}
|X_k| > \lambda_i, i=1,\ldots,d \Bigr)
\\
&&\quad\leq P \Biggl( C_{\alpha}^{1/\alpha} \bigvee
_{j=1}^{\infty} \Gamma _j^{-1/\alpha}
\frac{\max_{1 \leq k \leq\lfloor nt_i \rfloor
}|f_k(U_j^{(n)})|}{\max_{1 \leq k \leq n}|f_k(U_j^{(n)})|} > \lambda _i(1-\delta), i=1,\ldots,d \Biggr)
\\
&&\quad\quad{}+ \varphi_n \Bigl( C_{\alpha}^{-1/\alpha} \epsilon\min
_{1 \leq i \leq d} \lambda_i \Bigr) + \sum
_{i=1}^d \psi_n(\lambda_i,
t_i) ,
\end{eqnarray*}
where
\begin{eqnarray*}
\psi_n(\lambda, t) &=& P \Biggl( C_{\alpha}^{1/\alpha} \max
_{1 \leq k
\leq\lfloor nt \rfloor} \Biggl\llvert \sum_{j=1}^{\infty}
\epsilon_j \Gamma _j^{-1/\alpha} \frac{f_k(U_j^{(n)})}{\max_{1 \leq i \leq
n}|f_i(U_j^{(n)})|}
\Biggr\rrvert > \lambda,
\\
&&\hphantom{P \Biggl(}{}C_{\alpha}^{1/\alpha} \bigvee_{j=1}^{\infty}
\Gamma_j^{-1/\alpha} \frac
{\max_{1 \leq k \leq\lfloor nt \rfloor}|f_k(U_j^{(n)})|}{\max_{1
\leq
k \leq n}|f_k(U_j^{(n)})|} \leq\lambda(1-\delta) ,
\mbox{ and for each } m=1,\ldots,n,
\\
&&\hphantom{P \Biggl(}{}C_{\alpha}^{1/\alpha} \Gamma_j^{-1/\alpha}
\frac
{|f_m(U_j^{(n)})|}{\max_{1 \leq i \leq n}|f_i(U_j^{(n)})|} >
\epsilon\lambda \mbox{ for at most one } j=1,2,\ldots \Biggr) .
\end{eqnarray*}
By (\ref{e:single.poisson.jump}), it is enough to show that
%
\begin{equation}
\label{e:neg.term} \psi_n(\lambda, t) \to0
\end{equation}
for all $\lambda>0$ and $0 \leq t \leq1$.

For every $k=1,2,\ldots,n$, the Poisson random measure represented by
the points
\[
\Bigl( \epsilon_j \Gamma_j^{-1/\alpha}
f_k \bigl(U_j^{(n)} \bigr) \Bigl( \max
_{1
\leq i \leq n}\bigl|f_i \bigl(U_j^{(n)}
\bigr)\bigr| \Bigr)^{-1}, j=1,2,\ldots \Bigr) %
\]
has the same mean measure as that represented by the points
\[
\bigl( \epsilon_j \Gamma_j^{-1/\alpha} \| f
\|_{\alpha} b_n^{-1}, j=1,2,\ldots \bigr) , %
\]
where $\| f \|_{\alpha} =  ( \int_E |f|^{\alpha} \,\mathrm{d}\mu
)^{1/\alpha}$.
In fact, the common mean measure assigns the value $x^{-\alpha} \| f
\|_{\alpha}^{\alpha} / 2$ to the sets $(x,\infty)$ and $(-\infty, -x)$
for every $x>0$. Therefore, these two Poisson random measures coincide
distributionally. We conclude that the probability in (\ref
{e:neg.term}) is bounded by
\begin{eqnarray*}
&&\sum_{k=1}^{\lfloor nt \rfloor} P \Biggl(
C_{\alpha}^{1/\alpha} \Biggl\llvert \sum_{j=1}^{\infty}
\epsilon_j \Gamma_j^{-1/\alpha} \frac
{f_k(U_j^{(n)})}{\max_{1 \leq i \leq n}|f_i(U_j^{(n)})|}
\Biggr\rrvert > \lambda,%
\\
&&\hphantom{\sum_{k=1}^{\lfloor nt \rfloor} P \Biggl(}{} C_{\alpha}^{1/\alpha} \bigvee
_{j=1}^{\infty} \Gamma _j^{-1/\alpha}
\frac{f_k(U_j^{(n)})}{\max_{1 \leq i \leq
n}|f_i(U_j^{(n)})|} \leq\lambda(1-\delta) , %
\\
&&\hphantom{\sum_{k=1}^{\lfloor nt \rfloor} P \Biggl(}{}C_{\alpha}^{1/\alpha} \Gamma_j^{-1/\alpha}
\frac
{|f_k(U_j^{(n)})|}{\max_{1 \leq i \leq n}|f_i(U_j^{(n)})|} > \epsilon\lambda \mbox{ for at most one } j=1,2,\ldots \Biggr)
\\
&&\quad= \lfloor nt \rfloor P \Biggl( C_{\alpha}^{1/\alpha} \Biggl\llvert
\sum_{j=1}^{\infty} \epsilon_j
\Gamma_j^{-1/\alpha} \Biggr\rrvert > \lambda \| f
\|_{\alpha}^{-1} b_n , C_{\alpha}^{1/\alpha}
\bigvee_{j=1}^{\infty
} \Gamma_j^{-1/\alpha}
\leq\lambda(1-\delta) \| f \|_{\alpha}^{-1} b_n ,
\\
&&\hphantom{\quad= \lfloor nt \rfloor P \Biggl(}{}C_{\alpha}^{1/\alpha} \Gamma_j^{-1/\alpha} > \epsilon
\lambda\| f \| _{\alpha}^{-1} b_n \mbox{ for at most
one } j=1,2,\ldots \Biggr)
\\
&&\quad\leq n P \Biggl( C_{\alpha}^{1/\alpha} \Biggl\llvert \sum
_{j=K+1}^{\infty} \epsilon_j
\Gamma_j^{-1/\alpha} \Biggr\rrvert > (\delta- \epsilon K) \lambda
\| f \|_{\alpha}^{-1} b_n \Biggr)
\\
&&\quad\leq\frac{n \| f \|_{\alpha}^4 C_{\alpha}^{4/\alpha} }{(\delta-
\epsilon K)^4 \lambda^4 b_n^4} E \Biggl\llvert \sum_{j=K+1}^{\infty}
\epsilon _j \Gamma_j^{-1/\alpha} \Biggr\rrvert
^4 .
\end{eqnarray*}
Due to the choice $K+1>4/\alpha$,
\[
E \Biggl\llvert \sum_{j=K+1}^{\infty}
\epsilon_j \Gamma_j^{-1/\alpha} \Biggr\rrvert
^4 < \infty; %
\]
see Samorodnitsky \cite{samorodnitsky:2004a} for a detailed proof.
Since $n/b_n^4
\to0$ as $n \to\infty$, (\ref{e:neg.term}) follows.

Suppose now that $f = \one_A$. In that case, the probability measure
$\eta_n$ defined in \eqref{e:eta.n} coincides with the probability
measure $\mu_n$ of Proposition \ref{p:max.ergodic}.
In order to prove weak convergence in
the $J_1$-topology, we will use a truncation argument. We may and will
restrict ourselves to the space $D[0,1]$. Let
$K=1,2,\ldots$\,. First of all, we show, in the
notation of \eqref{e:represent.Z}, the convergence\looseness=1
%
\begin{eqnarray}
\label{e:fidi.truncate} &&\Biggl( C_{\alpha}^{1/\alpha} \max_{1 \leq k \leq\lfloor nt \rfloor}
\Biggl\llvert \sum_{j=1}^K
\epsilon_j \Gamma_j^{-1/\alpha} \one_A
\circ T^k \bigl(U_j^{(n)} \bigr) \Biggr\rrvert ,
0 \leq t\leq1 \Biggr)\nonumber
\\[-8pt]\\[-8pt]
&&\quad\Rightarrow \quad\Biggl( C_{\alpha}^{1/\alpha} \bigvee
_{j=1}^K \Gamma_j^{-1/\alpha}
\one_{\{ V_j \leq t \}}, 0\leq t\leq1 \Biggr)\nonumber
\end{eqnarray}
in the $J_1$-topology on $D[0,1]$. Indeed, by
\eqref{e:single.poisson.jump}, outside of an event of asymptotically
vanishing probability, the process in the left-hand side of
\eqref{e:fidi.truncate} is
%
\begin{equation}
\label{e:fidi.truncate1} \Biggl( C_{\alpha}^{1/\alpha} \bigvee
_{j=1}^K \Gamma_j^{-1/\alpha} \max
_{1 \leq k \leq\lfloor nt \rfloor} \one_A \circ T^k
\bigl(U_j^{(n)} \bigr), 0\leq t\leq1 \Biggr) .
\end{equation}
By Proposition \ref{p:max.ergodic}, we can put all the random
variables involved on the same probability space so that the time of
the single step in the $j$th term in \eqref{e:fidi.truncate1}
converges a.s. for each $j=1,\ldots, K$ to $V_j$. Then, trivially, the
process in \eqref{e:fidi.truncate1} converges a.s. in the
$J_1$-topology on $D[0,1]$ to the process in the right-hand side
of \eqref{e:fidi.truncate}. Therefore, the weak convergence in
\eqref{e:fidi.truncate} follows.\looseness=1

Next, we note that in the $J_1$-topology on the space $D[0,1]$,
\begin{eqnarray*}
&&\Biggl( C_{\alpha}^{1/\alpha} \bigvee_{j=1}^K
\Gamma_j^{-1/\alpha} \one_{\{ V_j \leq t \}}, 0\leq t\leq1 \Biggr)
\\
&&\quad\to\quad \Biggl( C_{\alpha}^{1/\alpha} \bigvee_{j=1}^{\infty}
\Gamma_j^{-1/\alpha} \one_{\{ V_j \leq t_i \}} 0\leq t\leq1 \Biggr)
\quad\quad\mbox{as } K \to\infty \mbox{ a.s.} %
\end{eqnarray*}
This is so because, as $K \to\infty$,
\begin{eqnarray*}
&&\sup_{0 \leq t \leq1} \Biggl( \bigvee_{j=1}^{\infty}
\Gamma_j^{-1/\alpha} \one_{\{ V_j \leq t \}} - \bigvee
_{j=1}^K \Gamma_j^{-1/\alpha}
\one_{\{ V_j \leq t \}} \Biggr)
\\
&&\quad\leq\Gamma_{K+1}^{-1/\alpha} \to0
\quad\quad\mbox{a.s.}
\end{eqnarray*}

According to Theorem 3.2 in Billingsley \cite{billingsley:1999}, the
$J_1$-convergence in \eqref{e:weak.conv} will follow once we show that
\[
\lim_{K \to\infty} \limsup_{n \to\infty} P \Biggl( \max
_{1 \leq k
\leq n} \Biggl\llvert \sum_{j=K+1}^{\infty}
\epsilon_j \Gamma_j^{-1/\alpha} \one_A
\circ T^k \bigl(U_j^{(n)} \bigr) \Biggr\rrvert >
\epsilon \Biggr) = 0 %
\]
for every $\epsilon>0$. Write
\begin{eqnarray*}
&&P \Biggl( \max_{1 \leq k \leq n} \Biggl\llvert \sum
_{j=K+1}^{\infty} \epsilon_j
\Gamma_j^{-1/\alpha} \one_A \circ T^k
\bigl(U_j^{(n)} \bigr) \Biggr\rrvert > \epsilon \Biggr)
\\
&&\quad\leq\int_0^{(\epsilon/2)^{-\alpha}} \mathrm{e}^{-x}
\frac{x^{K-1}}{(K-1)!} \,\mathrm{d}x
\\
&&\hphantom{\quad\leq}{} + \int_{(\epsilon/2)^{-\alpha}}^{\infty}
\mathrm{e}^{-x} \frac{x^{K-1}}{(K-1)!}
\\
&&\hphantom{\quad\leq+ \int_{(\epsilon/2)^{-\alpha}}^{\infty}}{}
\times P \Biggl( \max_{1 \leq k \leq n} \Biggl
\llvert \sum_{j=1}^{\infty}
\epsilon_j ( \Gamma_j + x )^{-1/\alpha}
\one_A \circ T^k \bigl(U_j^{(n)}
\bigr) \Biggr\rrvert > \epsilon \Biggr) \,\mathrm{d}x .
\end{eqnarray*}
Clearly, the first term vanishes when $K \to\infty$. Therefore, it
is sufficient to show that for every $x \geq(\epsilon/2)^{-\alpha}$,
%
\begin{equation}
\label{e:upper.app.Gamma1} P \Biggl( \max_{1 \leq k \leq n} \Biggl\llvert \sum
_{j=1}^{\infty} \epsilon_j (
\Gamma_j + x )^{-1/\alpha} \one_A \circ
T^k \bigl(U_j^{(n)} \bigr) \Biggr\rrvert >
\epsilon \Biggr) \to0
\end{equation}
as $n \to\infty$.

To this end, choose $L \in\bbn$ and $0 < \xi< 1/2$ so that
%
\begin{equation}
\label{e:const.rest2} L+1 > \frac{4}{\alpha} \quad\mbox{and}\quad \frac{1}{2} - \xi L >
0 .
\end{equation}
By \eqref{e:single.poisson.jump}, we can write
%
\begin{eqnarray}\label{e:one.term.vanish}
&&\hspace*{-18pt}P \Biggl( \max_{1 \leq k \leq n} \Biggl\llvert \sum
_{j=1}^{\infty} \epsilon_j (
\Gamma_j + x )^{-1/\alpha} \one_A \circ
T^k \bigl(U_j^{(n)} \bigr) \Biggr\rrvert >
\epsilon \Biggr) \nonumber
\\
&&\hspace*{-18pt}\quad\leq P \Biggl( \max_{1 \leq k \leq n} \Biggl\llvert \sum
_{j=1}^{\infty} \epsilon_j (
\Gamma_j + x )^{-1/\alpha} \one_A \circ
T^k \bigl(U_j^{(n)} \bigr) \Biggr\rrvert >
\epsilon, \mbox{ and for each } m = 1,\ldots, n,
\\
&&\hspace*{-18pt}\hphantom{\quad\leq P \Biggl(}{} ( \Gamma_j + x )^{-1/\alpha} \one_A \circ
T^m \bigl(U_j^{(n)} \bigr) > \xi\epsilon
\mbox{ for at most one } j =1,2,\ldots \Biggr) + \mathrm{o}(1) .
\nonumber
\end{eqnarray}
Notice that for every $k=1,\ldots, n$, the Poisson random measure
represented by the points
\[
\bigl( \epsilon_j ( \Gamma_j + x )^{-1/\alpha}
\one_A \circ T^k \bigl(U_j^{(n)}
\bigr), j=1,2,\ldots \bigr) %
\]
is distributionally equal to the Poisson random measure represented by
the points
\[
\bigl( \epsilon_j \bigl( b_n^{\alpha}
\mu(A)^{-1} \Gamma_j +x \bigr)^{-1/\alpha}, j=1,2,\ldots
\bigr) . %
\]
Therefore, the first term on the right-hand side of (\ref
{e:one.term.vanish}) can be bounded by
\begin{eqnarray*}
&&\sum_{k=1}^n P \Biggl( \Biggl\llvert \sum
_{j=1}^{\infty} \epsilon_j (
\Gamma_j + x )^{-1/\alpha} \one_A \circ
T^k \bigl(U_j^{(n)} \bigr) \Biggr\rrvert >
\epsilon,
\\
&&\hphantom{\sum_{k=1}^n P \Biggl(}{}( \Gamma_j + x )^{-1/\alpha} \one_A \circ
T^k \bigl(U_j^{(n)} \bigr) > \xi\epsilon
\mbox{ for at most one } j=1,2,\ldots \Biggr) %
\\
&&\quad= nP \Biggl( \Biggl\llvert \sum_{j=1}^{\infty}
\epsilon_j \bigl( b_n^{\alpha}
\mu(A)^{-1} \Gamma_j + x \bigr)^{-1/\alpha} \Biggr
\rrvert > \epsilon,
\\
&&\hphantom{\quad= nP \Biggl(}{}\bigl( b_n^{\alpha} \mu(A)^{-1}
\Gamma_j + x \bigr)^{-1/\alpha} > \xi \epsilon \mbox{ for at most
one } j = 1,2,\ldots \Biggr)
\\
&&\quad\leq n P \Biggl( \Biggl\llvert \sum_{j=L+1}^{\infty}
\epsilon_j \bigl( b_n^{\alpha}
\mu(A)^{-1} \Gamma_j + x \bigr)^{-1/\alpha} \Biggr
\rrvert > \biggl( \frac{1}{2} - \xi L \biggr)\epsilon \Biggr) . %
\end{eqnarray*}

In the last step we used the fact that, for $x \geq
(\epsilon/2)^{-\alpha}$, the magnitude of each term in the infinite
sum does not exceed $\epsilon/2$. By the contraction inequality for
Rademacher series (see, e.g., Proposition 1.2.1 of
Kwapie\'n and Woyczy\'nski \cite{kwapien:woyczynski:1992}),
\begin{eqnarray*}
&&nP \Biggl( \Biggl\llvert \sum_{j=L+1}^{\infty}
\epsilon_j \bigl( b_n^{\alpha}
\mu(A)^{-1} \Gamma_j + x \bigr)^{-1/\alpha} \Biggr
\rrvert > \biggl( \frac
{1}{2} - \xi L \biggr)\epsilon \Biggr) %
\\
&&\quad\leq2 n P \Biggl( \Biggl\llvert \sum_{j=L+1}^{\infty}
\epsilon_j \Gamma _j^{-1/\alpha} \Biggr\rrvert >
\biggl( \frac{1}{2} - \xi L \biggr)\epsilon\mu (A)^{-1/\alpha}
b_n \Biggr) . %
\end{eqnarray*}
As before, by Markov's inequality and using the
constraints of the constants $L \in\bbn$ and $0 < \xi< 1/2$ given in
(\ref{e:const.rest2}),
\begin{eqnarray*}
&&2n P \Biggl( \Biggl\llvert \sum_{j=L+1}^{\infty}
\epsilon_j \Gamma _j^{-1/\alpha
} \Biggr\rrvert >
\biggl( \frac{1}{2} - \xi L \biggr)\epsilon\mu (A)^{-1/\alpha
}
b_n \Biggr) \\
&&\quad\leq\frac{2n\mu(A)^{4/\alpha}}{(2^{-1}-\xi L)^4
\epsilon
^4 b_n^4} E \Biggl\llvert \sum
_{j=L+1}^{\infty} \epsilon_j \Gamma
_j^{-1/\alpha} \Biggr\rrvert ^4 \to0 %
\end{eqnarray*}
as $n \to\infty$ and, hence, (\ref{e:upper.app.Gamma1}) follows.
\end{pf*}

%
\begin{remark} \label{rk:other.beta}
The crucial point in the proof of the theorem is the ``single Poisson
jump property'' \eqref{e:single.poisson.jump} that shows that,
essentially, a single Poisson point of the type $\Gamma_j^{-1/\alpha}$
plays the decisive role in determining the size of a partial
maximum. This enabled us to show that the normalized partial maxima
converge to the first Poisson point $\Gamma_1^{-1/\alpha}$ (which, of
course, has exactly the standard \mbox{$\alpha$-Fr\'echet} law). We can
guarantee the ``single Poisson jump property'' in the case
$1/2<\beta<1$. On the other hand, in the range $0 < \beta< 1/2$,
the condition \eqref{e:single.poisson.jump} is no longer valid. We
believe that the limiting process will involve a finite, but
random, number of the Poisson points of the type
$\Gamma_j^{-1/\alpha}$. This will preclude a limiting Fr\'echet
law. The details of this are still being worked out, and will appear
in a future work. In the boundary case $\beta=1/2$, the statement
\eqref{e:weak.conv} still holds under certain additional
conditions. This is the case, for example, for the Markov shift
operators presented at the end of the paper. See also Example 5.3 in
Samorodnitsky \cite{samorodnitsky:2004a}.
\end{remark}

%
\begin{remark}
There is no doubt that the convergence result in Theorem
\ref{t:main.max} can be extended to more general infinitely divisible
random measures $M$ in \eqref{e:underlying.proc}, under appropriate
assumptions of regular variation of the L\'evy measure of $M$ and
integrability of the function $f$. In particular, regardless of the
size of $\alpha>0$, the time scaled extremal Fr\'echet processes
$Z_{\alpha, \beta}$ are likely to
appear in the limit in \eqref{e:weak.conv}. Furthermore, the symmetry
of the
process $\BX$ has very little to do with the limiting distribution of
the partial maxima. For example, a straightforward symmetrization
argument allows one to extend \eqref{e:weak.conv} to skewed
$\alpha$-stable processes, at least in the sense of convergence of
finite-dimensional distributions. The reason we decided to restrict
the presentation to the symmetric stable case had to do with a
particularly simple form of the series representation
\eqref{e:series.max} available in this case. This has allowed us to avoid
certain technicalities that might have otherwise blurred the main
message, which
is the effect of memory on the functional limit theorem for the
partial maxima.
\end{remark}

One can obtain concrete examples of the situations in which the result
of Theorem \ref{t:main.max} applies by taking, for instance, one of
the variety of pointwise
dual ergodic operators provided in Aaronson \cite{aaronson:1981} and
Zweim\"uller \cite{zweimuller:2009}, and embedding them into the
integral form of stationary S$\alpha$S processes. We conclude the
current paper
by mentioning the example of a flow generated by a
null recurrent Markov chain. This example appears in
Samorodnitsky \cite{samorodnitsky:2004a}, Owada and Samorodnitsky
\cite
{owada:samorodnitsky:2012}, and
Owada \cite{owada:2013} as well.

Consider an irreducible null recurrent Markov chain $(x_n, n \geq
0)$ defined on an infinite countable state space $\bbs$ with the
transition matrix $(p_{ij})$. Let $(\pi_i, i \in\bbs)$ be its
unique (up to constant multiplication) invariant measure with
$\pi_{i_0}=1$ for some fixed state $i_0\in\bbs$.
Note that $(\pi_i)$ is necessarily an infinite
measure. Define a $\sigma$-finite and infinite measure on
$(E,\mathcal{E}) = (\bbs^{\bbn}, \mathcal{B}(\bbs^{\bbn}))$ by
\[
\mu(B) = \sum_{i \in\bbs} \pi_i
P_i(B),\quad\quad B \subseteq\bbs^{\bbn} , %
\]
where $P_i(\cdot)$ denotes the probability law of $(x_n)$ starting in
state $i \in\bbs$. Let
\[
T(x_0, x_1, \ldots) = (x_1,x_2,
\ldots) %
\]
be the usual left shift operator on $\bbs^{\bbn}$. Then $T$ preserves
$\mu$. Since the Markov chain is irreducible and null recurrent, $T$
is conservative and ergodic (see Harris and Robbins \cite
{harris:robbins:1953}).

We consider the set $A = \{ x \in\bbs^{\bbn}\dvtx  x_0=i_0 \}$ with the
fixed state $i_0\in\bbs$ chosen above. Since
\[
\widehat{T}^k \one_A (x) = P_{i_0}
(x_k = i_0) \quad\quad\mbox{for } x \in A %
\]
is constant on $A$ (see Section~4.5 in Aaronson \cite{aaronson:1997}),
we can
choose as the normalizing sequence $a_n = \sum_{k=1}^n
P_{i_0}(x_k=i_0)$, and see that the expression $a_n^{-1} \sum_{k=1}^n
\widehat{T}^k \one_A(x)$ is identically equal to $1=\mu(A)$ on $A$.
Therefore, the map $T$ is pointwise dual ergodic, and the
Darling--Kac set condition, in fact, reduces to a simple
identity. Let
\[
\varphi_A(x) = \min\{ n \geq1\dvtx  x_n \in A \} ,\quad\quad x \in
\bbs^{\bbn} %
\]
be the first entrance time, and assume that
%
\begin{equation}
\label{e:return.DK} \sum_{k=1}^n
P_{i_0}(\varphi_A \geq k) \in RV_{\beta}
\end{equation}
for some $\beta\in(1/2, 1)$. Two equivalent conditions to
\eqref{e:return.DK} are given in
Resnick \textit{et~al}. \cite{resnick:samorodnitsky:xue:2000}. Note that the
exponent of regular variation $\beta$ controls how frequently the
Markov chain
returns to $A$. Since $\mu(\varphi_A=k) = P_{i_0}(\varphi_A \geq k)$
for $k \geq1$ (see Lemma 3.3 in
Resnick \textit{et~al}. \cite{resnick:samorodnitsky:xue:2000}), we have
\[
w_n \sim\mu(\varphi_A \leq n) \in RV_{\beta} .
\]

Then all of the assumptions of Theorem
\ref{t:main.max} are satisfied for any $f\in
L^{\alpha}(\mu)\cap L^{\infty}(\mu)$, supported by~$A$.


\section*{Acknowledgements}
We are very grateful for the uncommonly detailed
and useful comments we have received from two anonymous referees and
an anonymous Associate Editor. These comments helped us to introduce a
number of improvements to the
paper (this includes eliminating a serious error).
This research was partially supported by the ARO
Grants W911NF-07-1-0078 and W911NF-12-10385, NSF Grant DMS-1005903
and NSA Grant H98230-11-1-0154 at Cornell University.


%

\printhistory


\begin{thebibliography}{43}

\bibitem{aaronson:1981}
%
\begin{barticle}[mr]
\bauthor{\bsnm{Aaronson},~\bfnm{Jon}\binits{J.}}
(\byear{1981}).
\btitle{The asymptotic distributional behaviour of transformations
preserving infinite measures}.
\bjournal{J. Anal. Math.}
\bvolume{39}
\bpages{203--234}.
\bid{doi={10.1007/BF02803336}, issn={0021-7670}, mr={0632462}}
\end{barticle}
%
\bptok{imsref}%
\endbibitem

\bibitem{aaronson:1997}
%
\begin{bbook}[mr]
\bauthor{\bsnm{Aaronson},~\bfnm{Jon}\binits{J.}}
(\byear{1997}).
\btitle{An Introduction to Infinite Ergodic Theory}.
\bseries{Mathematical Surveys and Monographs}
\bvolume{50}.
\blocation{Providence, RI}:
\bpublisher{Amer. Math. Soc.}
\bid{mr={1450400}}
\end{bbook}
%
\bptok{imsref}%
\endbibitem

\bibitem{avram:taqqu:1992}
%
\begin{barticle}[mr]
\bauthor{\bsnm{Avram},~\bfnm{Florin}\binits{F.}} \AND
\bauthor{\bsnm{Taqqu},~\bfnm{Murad~S.}\binits{M.S.}}
(\byear{1992}).
\btitle{Weak convergence of sums of moving averages in the {$\alpha
$}-stable domain of attraction}.
\bjournal{Ann. Probab.}
\bvolume{20}
\bpages{483--503}.
\bid{issn={0091-1798}, mr={1143432}}
\end{barticle}
%
\bptok{imsref}%
\endbibitem

\bibitem{billingsley:1999}
%
\begin{bbook}[mr]
\bauthor{\bsnm{Billingsley},~\bfnm{Patrick}\binits{P.}}
(\byear{1999}).
\btitle{Convergence of Probability Measures},
\bedition{2nd} ed.
\bseries{Wiley Series in Probability and Statistics: Probability and
Statistics}.
\blocation{New York}:
\bpublisher{Wiley}.
\bid{doi={10.1002/9780470316962}, mr={1700749}}
\end{bbook}
%
\bptok{imsref}%
\endbibitem

\bibitem{davis:resnick:1985}
%
\begin{barticle}[mr]
\bauthor{\bsnm{Davis},~\bfnm{Richard}\binits{R.}} \AND
\bauthor{\bsnm{Resnick},~\bfnm{Sidney}\binits{S.}}
(\byear{1985}).
\btitle{Limit theory for moving averages of random variables with
regularly varying tail probabilities}.
\bjournal{Ann. Probab.}
\bvolume{13}
\bpages{179--195}.
\bid{issn={0091-1798}, mr={0770636}}
\end{barticle}
%
\bptok{imsref}%
\endbibitem

\bibitem{dehaan:ferreira:2006}
%
\begin{bbook}[mr]
\bauthor{\bparticle{de} \bsnm{Haan},~\bfnm{Laurens}\binits{L.}}
\AND
\bauthor{\bsnm{Ferreira},~\bfnm{Ana}\binits{A.}}
(\byear{2006}).
\btitle{Extreme Value Theory}.
\bseries{An Introduction. Springer Series in Operations Research and Financial Engineering}.
\blocation{New York}:
\bpublisher{Springer}.
\bid{mr={2234156}}
\end{bbook}
%
\bptok{imsref}%
\endbibitem

\bibitem{dwass:1964}
%
\begin{barticle}[mr]
\bauthor{\bsnm{Dwass},~\bfnm{Meyer}\binits{M.}}
(\byear{1964}).
\btitle{Extremal processes}.
\bjournal{Ann. Math. Statist.}
\bvolume{35}
\bpages{1718--1725}.
\bid{issn={0003-4851}, mr={0177440}}
\end{barticle}
%
\bptok{imsref}%
\endbibitem

\bibitem{dwass:1966}
%
\begin{barticle}[mr]
\bauthor{\bsnm{Dwass},~\bfnm{Meyer}\binits{M.}}
(\byear{1966}).
\btitle{Extremal processes. {II}}.
\bjournal{Illinois J. Math.}
\bvolume{10}
\bpages{381--391}.
\bid{issn={0019-2082}, mr={0193661}}
\end{barticle}
%
\bptok{imsref}%
\endbibitem

\bibitem{embrechts:maejima:2002}
%
\begin{bbook}[mr]
\bauthor{\bsnm{Embrechts},~\bfnm{Paul}\binits{P.}} \AND
\bauthor{\bsnm{Maejima},~\bfnm{Makoto}\binits{M.}}
(\byear{2002}).
\btitle{Selfsimilar Processes}.
\bseries{Princeton Series in Applied Mathematics}.
\blocation{Princeton, NJ}:
\bpublisher{Princeton Univ. Press}.
\bid{mr={1920153}}
\end{bbook}
%
\bptok{imsref}%
\endbibitem

\bibitem{fasen:2005}
%
\begin{barticle}[mr]
\bauthor{\bsnm{Fasen},~\bfnm{Vicky}\binits{V.}}
(\byear{2005}).
\btitle{Extremes of regularly varying {L}\'evy-driven mixed moving
average processes}.
\bjournal{Adv. in Appl. Probab.}
\bvolume{37}
\bpages{993--1014}.
\bid{doi={10.1239/aap/1134587750}, issn={0001-8678}, mr={2193993}}
\end{barticle}
%
\bptok{imsref}%
\endbibitem

\bibitem{fisher:tippett:1928}
%
\begin{barticle}[auto:STB|2014/02/12|14:17:21]
\bauthor{\bsnm{Fisher},~\bfnm{R.~A.}\binits{R.A.}} \AND
\bauthor{\bsnm{Tippett},~\bfnm{L.}\binits{L.}}
(\byear{1928}).
\btitle{Limiting forms of the frequency distributions of the largest
or smallest member of a sample}.
\bjournal{Math. Proc. Cambridge Philos. Soc.}
\bvolume{24}
\bpages{180--190}.
\end{barticle}
%
\bptok{imsref}%
\endbibitem

\bibitem{harris:robbins:1953}
%
\begin{barticle}[mr]
\bauthor{\bsnm{Harris},~\bfnm{T.~E.}\binits{T.E.}} \AND
\bauthor{\bsnm{Robbins},~\bfnm{Herbert}\binits{H.}}
(\byear{1953}).
\btitle{Ergodic theory of {M}arkov chains admitting an infinite
invariant measure}.
\bjournal{Proc. Natl. Acad. Sci. USA}
\bvolume{39}
\bpages{860--864}.
\bid{issn={0027-8424}, mr={0056873}}
\end{barticle}
%
\bptok{imsref}%
\endbibitem

\bibitem{kabluchko:schlather:dehaan:2009}
%
\begin{barticle}[mr]
\bauthor{\bsnm{Kabluchko},~\bfnm{Zakhar}\binits{Z.}},
\bauthor{\bsnm{Schlather},~\bfnm{Martin}\binits{M.}} \AND
\bauthor{\bparticle{de} \bsnm{Haan},~\bfnm{Laurens}\binits{L.}}
(\byear{2009}).
\btitle{Stationary max-stable fields associated to negative definite
functions}.
\bjournal{Ann. Probab.}
\bvolume{37}
\bpages{2042--2065}.
\bid{doi={10.1214/09-AOP455}, issn={0091-1798}, mr={2561440}}
\end{barticle}
%
\bptok{imsref}%
\endbibitem

\bibitem{kabluchko:stoev:2012}
%
\begin{bmisc}[auto:STB|2014/02/12|14:17:21]
\bauthor{\bsnm{Kabluchko},~\bfnm{Z.}\binits{Z.}} \AND
\bauthor{\bsnm{Stoev},~\bfnm{S.}\binits{S.}}
(\byear{2015}).
\bhowpublished{Stochastic integral representations and classification
of sum- and max-infinitely divisible processes. \textit{Bernoulli}.
 To appear.
 Available at \arxivurl{arXiv:1207.4983}}.
\end{bmisc}
%
\bptok{imsref}%
\endbibitem

\bibitem{krengel:1985}
%
\begin{bbook}[mr]
\bauthor{\bsnm{Krengel},~\bfnm{Ulrich}\binits{U.}}
(\byear{1985}).
\btitle{Ergodic Theorems}.
\bseries{de Gruyter Studies in Mathematics}
\bvolume{6}.
\blocation{Berlin}:
\bpublisher{de Gruyter}.
\bid{doi={10.1515/9783110844641}, mr={0797411}}
\end{bbook}
%
\bptok{imsref}%
\endbibitem

\bibitem{kwapien:woyczynski:1992}
%
\begin{bbook}[mr]
\bauthor{\bsnm{Kwapie{\'n}},~\bfnm{Stanis{\l}aw}\binits{S.}} \AND
\bauthor{\bsnm{Woyczy{\'n}ski},~\bfnm{Wojbor~A.}\binits{W.A.}}
(\byear{1992}).
\btitle{Random Series and Stochastic Integrals: Single and Multiple}.
\bseries{Probability and Its Applications}.
\blocation{Boston, MA}:
\bpublisher{Birkh\"auser}.
\bid{doi={10.1007/978-1-4612-0425-1}, mr={1167198}}
\end{bbook}
%
\bptok{imsref}%
\endbibitem

\bibitem{lamperti:1962}
%
\begin{barticle}[mr]
\bauthor{\bsnm{Lamperti},~\bfnm{John}\binits{J.}}
(\byear{1962}).
\btitle{Semi-stable stochastic processes}.
\bjournal{Trans. Amer. Math. Soc.}
\bvolume{104}
\bpages{62--78}.
\bid{issn={0002-9947}, mr={0138128}}
\end{barticle}
%
\bptok{imsref}%
\endbibitem

\bibitem{lamperti:1964}
%
\begin{barticle}[mr]
\bauthor{\bsnm{Lamperti},~\bfnm{John}\binits{J.}}
(\byear{1964}).
\btitle{On extreme order statistics}.
\bjournal{Ann. Math. Statist.}
\bvolume{35}
\bpages{1726--1737}.
\bid{issn={0003-4851}, mr={0170371}}
\end{barticle}
%
\bptok{imsref}%
\endbibitem

\bibitem{leadbetter:1983}
%
\begin{barticle}[mr]
\bauthor{\bsnm{Leadbetter},~\bfnm{M.~R.}\binits{M.R.}}
(\byear{1983}).
\btitle{Extremes and local dependence in stationary sequences}.
\bjournal{Z. Wahrsch. Verw. Gebiete}
\bvolume{65}
\bpages{291--306}.
\bid{doi={10.1007/BF00532484}, issn={0044-3719}, mr={0722133}}
\end{barticle}
%
\bptok{imsref}%
\endbibitem

\bibitem{leadbetter:lindgren:rootzen:1983}
%
\begin{bbook}[mr]
\bauthor{\bsnm{Leadbetter},~\bfnm{M.~R.}\binits{M.R.}},
\bauthor{\bsnm{Lindgren},~\bfnm{Georg}\binits{G.}} \AND
\bauthor{\bsnm{Rootz{\'e}n},~\bfnm{Holger}\binits{H.}}
(\byear{1983}).
\btitle{Extremes and Related Properties of Random Sequences and Processes}.
\bseries{Springer Series in Statistics}.
\blocation{New York}:
\bpublisher{Springer}.
\bid{mr={0691492}}
\end{bbook}
%
\bptok{imsref}%
\endbibitem

\bibitem{mikosch:starica:2000b}
%
\begin{barticle}[mr]
\bauthor{\bsnm{Mikosch},~\bfnm{Thomas}\binits{T.}} \AND
\bauthor{\bsnm{St{\u{a}}ric{\u{a}}},~\bfnm{C{\u{a}}t{\u
{a}}lin}\binits{C.}}
(\byear{2000}).
\btitle{Limit theory for the sample autocorrelations and extremes of a
$\operatorname{GARCH}(1,1)$ process}.
\bjournal{Ann. Statist.}
\bvolume{28}
\bpages{1427--1451}.
\bid{doi={10.1214/aos/1015957401}, issn={0090-5364}, mr={1805791}}
\end{barticle}
%
\bptok{imsref}%
\endbibitem

\bibitem{obrien:torfs:vervaat:1990}
%
\begin{barticle}[mr]
\bauthor{\bsnm{O'Brien},~\bfnm{George~L.}\binits{G.L.}},
\bauthor{\bsnm{Torfs},~\bfnm{Paul~J.~J.~F.}\binits{P.J.J.F.}} \AND
\bauthor{\bsnm{Vervaat},~\bfnm{Wim}\binits{W.}}
(\byear{1990}).
\btitle{Stationary self-similar extremal processes}.
\bjournal{Probab. Theory Related Fields}
\bvolume{87}
\bpages{97--119}.
\bid{doi={10.1007/BF01217748}, issn={0178-8051}, mr={1076958}}
\end{barticle}
%
\bptok{imsref}%
\endbibitem

\bibitem{owada:2013}
%
\begin{bmisc}[auto:STB|2014/02/12|14:17:21]
\bauthor{\bsnm{Owada},~\bfnm{T.}\binits{T.}}
(\byear{2013}).
\bhowpublished{Limit theory for the sample autocovariance for
heavy tailed stationary infinitely divisible processes
generated by conservative flows. Technical report.
Available at \arxivurl{arXiv:1302.0058}}.
\end{bmisc}
%
\bptok{imsref}%
\endbibitem

\bibitem{owada:samorodnitsky:2012}
%
\begin{barticle}[auto:STB|2014/02/12|14:17:21]
\bauthor{\bsnm{Owada},~\bfnm{T.}\binits{T.}} \AND
\bauthor{\bsnm{Samorodnitsky},~\bfnm{G.}\binits{G.}}
(\byear{2015}).
\btitle{Functional central limit theorem for heavy
tailed stationary ifinitely divisible processes generated
by conservative flows}.
\bjournal{Ann. Probab.}
\bvolume{43}
\bpages{240--285}.
\bid{mr={3298473}}
\end{barticle}
%
\bptok{imsref}%
\endbibitem

\bibitem{resnick:samorodnitsky:2004}
%
\begin{barticle}[mr]
\bauthor{\bsnm{Resnick},~\bfnm{Sidney}\binits{S.}} \AND
\bauthor{\bsnm{Samorodnitsky},~\bfnm{Gennady}\binits{G.}}
(\byear{2004}).
\btitle{Point processes associated with stationary stable processes}.
\bjournal{Stochastic Process. Appl.}
\bvolume{114}
\bpages{191--209}.
\bid{doi={10.1016/j.spa.2004.06.004}, issn={0304-4149}, mr={2101240}}
\end{barticle}
%
\bptok{imsref}%
\endbibitem

\bibitem{resnick:samorodnitsky:xue:2000}
%
\begin{barticle}[mr]
\bauthor{\bsnm{Resnick},~\bfnm{Sidney}\binits{S.}},
\bauthor{\bsnm{Samorodnitsky},~\bfnm{Gennady}\binits{G.}} \AND
\bauthor{\bsnm{Xue},~\bfnm{Fang}\binits{F.}}
(\byear{2000}).
\btitle{Growth rates of sample covariances of stationary symmetric
{$\alpha$}-stable processes associated with null recurrent {M}arkov chains}.
\bjournal{Stochastic Process. Appl.}
\bvolume{85}
\bpages{321--339}.
\bid{doi={10.1016/S0304-4149(99)00081-2}, issn={0304-4149}, mr={1731029}}
\end{barticle}
%
\bptok{imsref}%
\endbibitem

\bibitem{resnick:1987}
%
\begin{bbook}[mr]
\bauthor{\bsnm{Resnick},~\bfnm{Sidney~I.}\binits{S.I.}}
(\byear{1987}).
\btitle{Extreme Values, Regular Variation, and Point Processes}.
\bseries{Applied Probability. A~Series of the Applied Probability Trust}
\bvolume{4}.
\blocation{New York}:
\bpublisher{Springer}.
\bid{mr={0900810}}
\end{bbook}
%
\bptok{imsref}%
\endbibitem

\bibitem{resnick:rubinovitch:1973}
%
\begin{barticle}[mr]
\bauthor{\bsnm{Resnick},~\bfnm{Sidney~I.}\binits{S.I.}} \AND
\bauthor{\bsnm{Rubinovitch},~\bfnm{Michael}\binits{M.}}
(\byear{1973}).
\btitle{The structure of extremal processes}.
\bjournal{Adv. in Appl. Probab.}
\bvolume{5}
\bpages{287--307}.
\bid{issn={0001-8678}, mr={0350867}}
\end{barticle}
%
\bptok{imsref}%
\endbibitem

\bibitem{rootzen:1978}
%
\begin{barticle}[mr]
\bauthor{\bsnm{Rootz{\'e}n},~\bfnm{Holger}\binits{H.}}
(\byear{1978}).
\btitle{Extremes of moving averages of stable processes}.
\bjournal{Ann. Probab.}
\bvolume{6}
\bpages{847--869}.
\bid{issn={0091-1798}, mr={0494450}}
\end{barticle}
%
\bptok{imsref}%
\endbibitem

\bibitem{rosinski:1995}
%
\begin{barticle}[mr]
\bauthor{\bsnm{Rosi{\'n}ski},~\bfnm{Jan}\binits{J.}}
(\byear{1995}).
\btitle{On the structure of stationary stable processes}.
\bjournal{Ann. Probab.}
\bvolume{23}
\bpages{1163--1187}.
\bid{issn={0091-1798}, mr={1349166}}
\end{barticle}
%
\bptok{imsref}%
\endbibitem

\bibitem{rosinski:2006}
%
\begin{barticle}[mr]
\bauthor{\bsnm{Rosi{\'n}ski},~\bfnm{Jan}\binits{J.}}
(\byear{2006}).
\btitle{Minimal integral representations of stable processes}.
\bjournal{Probab. Math. Statist.}
\bvolume{26}
\bpages{121--142}.
\bid{issn={0208-4147}, mr={2301892}}
\end{barticle}
%
\bptok{imsref}%
\endbibitem

\bibitem{roy:2007}
%
\begin{barticle}[mr]
\bauthor{\bsnm{Roy},~\bfnm{Emmanuel}\binits{E.}}
(\byear{2007}).
\btitle{Ergodic properties of {P}oissonian {ID} processes}.
\bjournal{Ann. Probab.}
\bvolume{35}
\bpages{551--576}.
\bid{doi={10.1214/009117906000000692}, issn={0091-1798}, mr={2308588}}
\end{barticle}
%
\bptok{imsref}%
\endbibitem

\bibitem{samorodnitsky:2004a}
%
\begin{barticle}[mr]
\bauthor{\bsnm{Samorodnitsky},~\bfnm{Gennady}\binits{G.}}
(\byear{2004}).
\btitle{Extreme value theory, ergodic theory and the boundary between
short memory and long memory for stationary stable processes}.
\bjournal{Ann. Probab.}
\bvolume{32}
\bpages{1438--1468}.
\bid{doi={10.1214/009117904000000261}, issn={0091-1798}, mr={2060304}}
\end{barticle}
%
\bptok{imsref}%
\endbibitem

\bibitem{samorodnitsky:2005}
%
\begin{barticle}[mr]
\bauthor{\bsnm{Samorodnitsky},~\bfnm{Gennady}\binits{G.}}
(\byear{2005}).
\btitle{Null flows, positive flows and the structure of stationary
symmetric stable processes}.
\bjournal{Ann. Probab.}
\bvolume{33}
\bpages{1782--1803}.
\bid{doi={10.1214/009117905000000305}, issn={0091-1798}, mr={2165579}}
\end{barticle}
%
\bptok{imsref}%
\endbibitem

\bibitem{samorodnitsky:2006LRD}
%
\begin{barticle}[mr]
\bauthor{\bsnm{Samorodnitsky},~\bfnm{Gennady}\binits{G.}}
(\byear{2006}).
\btitle{Long range dependence}.
\bjournal{Found. Trends Stoch. Syst.}
\bvolume{1}
\bpages{163--257}.
\bid{doi={10.1561/0900000004}, issn={1551-3106}, mr={2379935}}
\end{barticle}
%
\bptok{imsref}%
\endbibitem

\bibitem{samorodnitsky:taqqu:1994}
%
\begin{bbook}[mr]
\bauthor{\bsnm{Samorodnitsky},~\bfnm{Gennady}\binits{G.}} \AND
\bauthor{\bsnm{Taqqu},~\bfnm{Murad~S.}\binits{M.S.}}
(\byear{1994}).
\btitle{Stable Non-{G}aussian Random Processes}.
\bseries{Stochastic Modeling}.
\blocation{New York}:
\bpublisher{Chapman and Hall}.
\bid{mr={1280932}}
\end{bbook}
%
\bptok{imsref}%
\endbibitem

\bibitem{skorohod:1956}
%
\begin{barticle}[mr]
\bauthor{\bsnm{Skorohod},~\bfnm{A.~V.}\binits{A.V.}}
(\byear{1956}).
\btitle{Limit theorems for stochastic processes}.
\bjournal{Teor. Veroyatn. Primen.}
\bvolume{1}
\bpages{289--319}.
\bid{issn={0040-361X}, mr={0084897}}
\end{barticle}
%
\bptok{imsref}%
\endbibitem

\bibitem{stoev:2008}
%
\begin{barticle}[mr]
\bauthor{\bsnm{Stoev},~\bfnm{Stilian~A.}\binits{S.A.}}
(\byear{2008}).
\btitle{On the ergodicity and mixing of max-stable processes}.
\bjournal{Stochastic Process. Appl.}
\bvolume{118}
\bpages{1679--1705}.
\bid{doi={10.1016/j.spa.2007.10.013}, issn={0304-4149}, mr={2442375}}
\end{barticle}
%
\bptok{imsref}%
\endbibitem

\bibitem{stoev:taqqu:2005}
%
\begin{barticle}[mr]
\bauthor{\bsnm{Stoev},~\bfnm{Stilian~A.}\binits{S.A.}} \AND
\bauthor{\bsnm{Taqqu},~\bfnm{Murad~S.}\binits{M.S.}}
(\byear{2005}).
\btitle{Extremal stochastic integrals: A parallel between max-stable
processes and {$\alpha$}-stable processes}.
\bjournal{Extremes}
\bvolume{8}
\bpages{237--266}.
\bid{doi={10.1007/s10687-006-0004-0}, issn={1386-1999}, mr={2324891}}
\end{barticle}
%
\bptok{imsref}%
\endbibitem

\bibitem{surgailis:rosinski:mandrekar:cambanis:1993}
%
\begin{barticle}[mr]
\bauthor{\bsnm{Surgailis},~\bfnm{Donatas}\binits{D.}},
\bauthor{\bsnm{Rosi{\'n}ski},~\bfnm{Jan}\binits{J.}},
\bauthor{\bsnm{Mandrekar},~\bfnm{V.}\binits{V.}} \AND
\bauthor{\bsnm{Cambanis},~\bfnm{Stamatis}\binits{S.}}
(\byear{1993}).
\btitle{Stable mixed moving averages}.
\bjournal{Probab. Theory Related Fields}
\bvolume{97}
\bpages{543--558}.
\bid{doi={10.1007/BF01192963}, issn={0178-8051}, mr={1246979}}
\end{barticle}
%
\bptok{imsref}%
\endbibitem

\bibitem{wang:stoev:2010}
%
\begin{barticle}[mr]
\bauthor{\bsnm{Wang},~\bfnm{Yizao}\binits{Y.}} \AND
\bauthor{\bsnm{Stoev},~\bfnm{Stilian~A.}\binits{S.A.}}
(\byear{2010}).
\btitle{On the structure and representations of max-stable processes}.
\bjournal{Adv. in Appl. Probab.}
\bvolume{42}
\bpages{855--877}.
\bid{doi={10.1239/aap/1282924066}, issn={0001-8678}, mr={2779562}}
\end{barticle}
%
\bptok{imsref}%
\endbibitem

\bibitem{whitt:2002}
%
\begin{bbook}[mr]
\bauthor{\bsnm{Whitt},~\bfnm{Ward}\binits{W.}}
(\byear{2002}).
\btitle{Stochastic-Process Limits}.
\bseries{An Introduction to Stochastic-Process Limits and Their
Application to Queues. Springer Series in Operations Research}.
\blocation{New York}:
\bpublisher{Springer}.
\bid{mr={1876437}}
\end{bbook}
%
\bptok{imsref}%
\endbibitem

\bibitem{zweimuller:2009}
%
\begin{bmisc}[auto:STB|2014/02/12|14:17:21]
\bauthor{\bsnm{Zweim{\"u}ller},~\bfnm{R.}\binits{R.}}
(\byear{2009}).
\bhowpublished{Surrey notes on infinite ergodic theory.
Lecture notes, Surrey Univ.}
\end{bmisc}
%
\bptok{imsref}%
\endbibitem

\end{thebibliography}
\end{document}